\documentclass{amsart}
\usepackage{graphicx,amssymb,amsmath}
\usepackage{amsfonts,amsthm,color}
\usepackage{bbm}
\usepackage{subfigure}
\usepackage{psfrag}

\newtheorem{theorem}{Theorem}[section]
\newtheorem{corollary}[theorem]{Corollary}
\newtheorem{lemma}[theorem]{Lemma}
\newtheorem{proposition}[theorem]{Proposition}

\theoremstyle{definition}
\newtheorem{definition}[theorem]{Definition}
\newtheorem{example}[theorem]{Example}
\newtheorem{remark}{Remark}

\newcommand{\R}{\mathbb{R}}
\newcommand{\Borel}{\mathcal{B}}
\newcommand{\blocks}{\mathcal{C}}
\newcommand{\F}{\mathcal{F}}
\newcommand{\p}{\mathcal{L}}
\newcommand{\prob}{\mathbb{P}}
\newcommand{\base}{\vartheta}
\newcommand{\rset}{A}

\newcommand{\oset}{A}
\newcommand{\rmatrix}{M}
\newcommand{\submatrix}{Q}
\newcommand{\norm}[1]{\|#1\|}
\newcommand{\abs}[1]{|#1|}
\renewcommand{\mod}[1]{~(\text{mod}~#1)}

\DeclareMathOperator{\proj}{Pr}
\DeclareMathOperator{\essinf}{ess~inf}
\DeclareMathOperator{\bv}{BV}
\DeclareMathOperator{\var}{var}
\DeclareMathOperator{\End}{End}

\begin{document}
\title[Metastability, Lyapunov exponents, escape rates and entropy in RDS]{Metastability, Lyapunov exponents, escape rates, and topological entropy in random dynamical systems}
\author{Gary Froyland}
\thanks{GF is partially supported by the ARC Discovery Project DP110100068}
\author{Ognjen Stancevic}
\address{ School of Mathematics and Statistics \\ University of New South Wales, Sydney NSW 2052, Australia}
\thanks{OS is supported by the ARC Centre of Excellence for Mathematics and Statistics of Complex Systems (MASCOS)}



\begin{abstract}
 We explore the concept of metastability in random dynamical systems, focussing on connections between random Perron-Frobenius operator cocycles and escape rates of random maps, and on topological entropy of random shifts of finite type. The Lyapunov spectrum of the random Perron-Frobenius cocycle and the random adjacency matrix cocycle is used to decompose the random system into two disjoint random systems with rigorous upper and lower bounds on (i) the escape rate in the setting of random maps, and (ii) topological entropy in the setting of random shifts of finite type, respectively.
\end{abstract}
\maketitle
\section{Introduction}\label{sec:intro}
Metastability in dynamical systems refers to the existence of subdomains, known as \emph{metastable sets} or \emph{almost-invariant sets}, within which trajectories are confined for long periods of time.
Metastability can be driven by small amounts of noise, for example, in a gradient system $\dot{x}=-\nabla V(x)$ where $V(x)$ is a double well potential, the noise enables sample paths to move between the metastable wells \cite{GOV87}. Alternatively, metastability can arise from purely deterministic means if trajectories pass between two or more subdomains sufficiently infrequently.
Applications of metastability or almost-invariance include molecular dynamics \cite{SHD01}, where the metastable sets are stable molecular conformations; astrodynamics \cite{DJKWL05}, where the metastable sets are regions from which asteroid escape is rare;  physical oceanography \cite{FPET07, DFHPS09}, where metastable regions are stable structures such as gyres and eddies;  and atmospheric science \cite{SFM10,FSM10}, where vortices in the stratosphere form time-dependent metastable regions.


In the deterministic setting one considers a nonsingular map $T:X\to X$ on a smooth manifold $X$.
The Perron-Frobenius operator $\p=\p_T$ is the natural push-forward for densities $f:X\to\mathbb{R}$ and is a key tool for studying the metastability of $T$.
The existence of metastability or almost-invariance in deterministic systems is often linked to the existence of isolated spectral values of the Perron-Frobenius operator \cite{DJ99,DFS00}.
Fixed points of $\p$ are invariant densities for $T$, and in suitable Banach spaces, the presence of real isolated sub-unit eigenvalues of the Perron-Frobenius operator correspond to eigenfunctions whose decay rates are \emph{slower} than the exponential separation of nearby trajectories \cite{DFS00}.
These eigenfunctions have been used to heuristically decompose the domain into metastable regions, linking the slow exponential decay of the eigenfunctions with slow exchange of trajectories.
If the Perron-Frobenius operator is denoted by $\p$ and one has $\p f = \lambda f$ for an isolated $\lambda \in \R$, the simplest way to decompose the system domain is to partition it into two sets $A_+:=\{f\ge 0\}$ and $A_-:=\{f<0\}$., see e.g. \cite{DJ99}.
At the functional level, exchange of mass between these two sets corresponds to cancellation between the positive and negative parts of $f$.
Since $\lambda$ is close to 1, this cancellation is low and therefore the mass exchange is also low (ie.\ the measure of $A_+\cap T^{-1}A_-$ and $A_-\cap T^{-1}A_+$ is low).



The slow exponential decay of such eigenfunctions may also be linked with slow exponential \emph{escape}.
One may create an \emph{open} dynamical system by restricting the dynamics to a metastable set.
Trajectories will stay in the metastable set for a while, but eventually leave.
One may ask how the \emph{rate of escape} from a metastable set is related to the second (say) eigenvalue $\lambda\in(0,1)$ of the Perron-Frobenius operator $\p:L^1(X)\to L^1(X)$ acting on the full domain.
The authors answered this question in \cite{FS10}: if the metastable set is defined as $\{x\ : \ f(x) > 0\}$ where $\p f = \lambda f$, then the rate of escape from this set is slower than $-\log \lambda$.

The concept of metastability and almost invariant sets was extended to random dynamical systems in \cite{FLQ09}, where these sets are referred to as \emph{coherent structures}.
In the random setting, the metastable sets may depend on the current random configuration $\omega\in\Omega$;  the dynamics on $\Omega$ is driven by $\vartheta:\Omega\circlearrowleft$, and the dynamics on $X$ is now governed by a cocycle of maps $\cdots T_{\vartheta^2\omega}\circ T_{\vartheta\omega}\circ T_\omega$.
One seeks a family of sets $A_\omega,A_{\vartheta\omega},A_{\vartheta^2\omega},\ldots$ so that the size of $A_\omega\cap T_\omega^{-1}A_{\vartheta\omega}\cap T^{-1}_\omega\circ T_{\vartheta\omega}^{-1}A_{\vartheta^2\omega}\cdots$ decays slowly.
Thus, the family $A_\omega,A_{\vartheta\omega},A_{\vartheta^2\omega},\ldots$ remains approximately coherent:  starting in $A_\omega$, forward trajectories mostly fall in the sequence of sets $A_\omega,A_{\vartheta\omega},A_{\vartheta^2\omega},\ldots$.
In physical applications, $\omega$ may represent time, $\vartheta$ the passage of time, and $T_\omega$ the physical dynamics at time $\omega$.
As a geophysical example, $X$ is the upper part of the ocean and $T_\omega$ describes evolution of water particles over a period of one week during a particular week $\omega$.
The $A_\omega,A_{\vartheta\omega},A_{\vartheta^2\omega},\ldots$ may represent, for example, the locations of an oceanic eddy, which meanders about the ocean surface, carries particularly warm or cold water with it, and very slowly disperses over time (the dispersion corresponds to random trajectories eventually falling outside the sequence of sets $A_\omega,A_{\vartheta\omega},A_{\vartheta^2\omega},\ldots$).

To move from deterministic to random (or time-dependent) concepts of metastability, \cite{FLQ09} and \cite{FLQ10} introduced the \emph{Lyapunov spectrum} for cocycles of random Perron-Frobenius operators, replacing the spectrum of a single deterministic Perron-Frobenius operator $\p$.
One studies decay rates of the norms of functions under the random composition $\cdots\p_{\vartheta^2\omega}\circ\p_{\vartheta\omega}\circ \p_{\omega}$.
In certain settings \cite{FLQ09,FLQ10}, one can identify equivariant random subspaces $E(\omega)$ that contain functions whose norm decays at specific rates and satisfy $\p_\omega E(\omega)=E(\vartheta\omega)$.
These random \emph{Oseledets subspaces} $E(\omega)$ play the role of eigenfunctions when determining the random metastable sets.
Suppose that $E(\omega)$ is one-dimensional, and let $f_\omega\in E(\omega)$.
If $\lim_{n\to\infty}(1/n)\log\|\cdots\p_{\vartheta^2\omega}\circ\p_{\vartheta\omega}\circ \p_{\omega}f_\omega\|=\lambda$, and $\lambda$ is close to 0, the norm of the function $f_\omega$ decays slowly.
One may define $A_{\omega,+}:=\{f_\omega\ge 0\}$, and in analogy to the cancellation argument given above in the autonomous setting, one expects the family of sets $A_{\omega,+},A_{\vartheta\omega,+},A_{\vartheta^2\omega,+},\ldots$ to be such that the size of $A_{\omega,+}\cap T_\omega^{-1}A_{\vartheta\omega,+}\cap T^{-1}_\omega\circ T_{\vartheta\omega^{-1}}A_{\vartheta^2\omega,+}\cdots$ decays slowly, and therefore represents a coherent family of sets.
Numerical algorithms and experiments based on this Perron-Frobenius cocycle theory were detailed in \cite{FLS10}.

Our goal in this paper is to \emph{link the slow decay of random functions induced by the Perron-Frobenius cocycle with escape rates from random metastable sets}.
Studies of escape rates for random dynamical systems have, to our knowledge, largely been concerned with escape from \emph{fixed} ($\omega$-invariant) sets under random or randomly perturbed maps (see for example \cite{Han86,Gra89}; for more recent work see \cite{DG09, RGM10}).
Other recent work considers escape rates for iid compositions of expanding interval maps with small holes \cite{BV}.
In this paper we work with a more general concept of escape from a \emph{random} set under a \emph{random} map.
We extend the results of \cite{FS10} to random dynamical systems by showing a relationship between Lyapunov spectrum and the corresponding random escape rates from metastable sets.

More precisely, in Section \ref{sec:general}, we show in a rather general setting ($X$ measurable, nonsingular dynamics) that given a Lyapunov exponent $\lambda(\omega,f)\lesssim 0$ of $\omega\in\Omega$ and $f\in L^1$, one can define a metastable random set $A$ along the orbit of $\omega$ as $\rset(\base^n\omega) = \{\p^{(n)}_{\omega}f \ge 0\}$.
Our first main result states that the escape rate from $\rset$ is slower than $-\lambda(\omega,f)$.

In Section \ref{sec:oseledets} we extend these results to random dynamical systems that admit an Oseledets splitting and, in particular, to Rychlik random dynamical systems where the dynamics are given by random expanding piecewise $C^2$ interval maps and $\p$ acts on $\bv=\bv([0,1])$.
This setting has historically been a standard testbed for spectral analysis of chaotic dynamical systems.
For example \cite{HK82,Ryc83} proved that if $1/|T'(x)|$ has bounded variation and if $\varrho:=\lim_{n\to\infty}\left(\|1/|(T^n)'|\|_\infty\right)^{1/n}<1$ then $\p:\bv\to \bv$ is quasi-compact.
Soon after, Keller \cite{Kel84} proved that $\varrho$ is the essential spectral radius.
As $\varrho$ is intimately connected with the average expansion experienced along orbits, $\bv$ spectral points of $\p$ larger than $\varrho$ in magnitude cannot be explained by local expansion of $T$ and must be due to the influence of global structures such as almost-invariant and metastable sets, producing decay rates \emph{slower than the average local expansion rate}.
In the random setting, Froyland \emph{et al.} \cite{FLQ10} proved a result parallel to \cite{Kel84} for random Rychlik maps.
Our main result in Section \ref{sec:oseledets} relates the escape rate from random almost-invariant sets to isolated values in the Lyapunov spectrum of $\p$.



In Section \ref{sec:rsft} we adapt the techniques of Section \ref{sec:general} to partition a random shift of finite type into two disjoint random subshifts, each with a topological entropy that is large, relative to the topological entropy of the original shift.
The more metastable the random shift, the larger the topological entropy of the subshifts.
More precisely, we show how to constructively decompose a random shift $\Sigma_\rmatrix$ into two complementary random shifts $\Sigma_{\submatrix}, \Sigma_{\submatrix'}$ with disjoint alphabets so that $h(\Sigma_\submatrix(\omega^*)),h(\Sigma_{\submatrix'}(\omega^*))\ge \lambda_2$, where $\omega^*\in\Omega$ and $\lambda_2$ is the second largest Lyapunov exponent of the adjacency matrix cocycle for $\Sigma_\rmatrix$.




\section{A result on escape rate for a general random dynamical system}\label{sec:general}

We use the notation of Arnold \cite{Arn98}.
Let $\base$ be an invertible measure-preserving transformation of a probability space $(\Omega, \F, \prob)$. We will call $\base$ the \emph{base transformation} and the tuple $(\Omega,\F,\prob, \base)$ the \emph{base dynamical system}. For technical reasons we shall additionally assume that singletons of $\Omega$ are $\F$-measurable. By $\End(Z)$ we denote the set of endomorphisms (or transformations) of a space $Z$, which preserves whatever structure $Z$ may have (e.g.~linearity or measurability). A \emph{random dynamical system} is a tuple $(\Omega, \F, \prob, \base, Z, \tilde\Phi)$ where $\tilde\Phi:\Omega \to \End(Z)$ is a family of transformations of $Z$ indexed by $\Omega$. A random dynamical system defines a discrete \emph{cocycle} $\Phi:\mathbb{N}\times \Omega \to \text{End}(Z)$ by
\begin{equation}\label{eq:cocycle}
	\Phi(n,\omega) = \Phi^{(n)}(\omega) = \Phi_{\omega}^{(n)} := \tilde\Phi(\base^{n-1}\omega)\circ \cdots \circ \tilde\Phi(\base\omega)\circ\tilde\Phi(\omega)
\end{equation}
with the property
\begin{itemize}
	\item[(i)] $\Phi(0,\omega) = id$;
	\item[(ii)] $\forall m,n \in \mathbb{N}, \Phi(m+n,\omega) = \Phi(m,\base^n\omega) \circ \Phi(n,\omega)$.
\end{itemize}
The function $\Phi$ is referred to as the cocycle while $\tilde\Phi$ is called the \emph{generator}. We will study two types of cocycles: \emph{measurable map cocycles} and their \emph{Perron-Frobenius operator cocycles}.

Firstly, let us assume that $Z$ is measure space $(X,\Borel,m)$ where $\Borel$ is its $\sigma$-algebra and $m$ a finite measure. Also let $\Borel(\mathbb{N})$ be the Borel $\sigma$-algebra on $\mathbb{N}$ (with respect to the discrete topology). A \emph{measurable (map) cocycle} is a cocycle $T:\mathbb{N}\times \Omega \to \End(X)$ such that the mapping $(n,\omega, x) \mapsto T_{\omega}^{(n)}(x)$ is $(\Borel(\mathbb{N})\otimes\mathcal{F}\otimes\Borel, \Borel)$-measurable and each $T_{\omega}:X \to X$ is non-singular ($m\circ T_{\omega}^{-1}\ll m$).



\subsection{Relating escape rate to Lyapunov exponents}\label{sec:mainthm}

Define a \emph{random set}\footnote{Our definition of a random set is slightly weaker than Arnold's \cite{Arn98} definition of a closed random set, where $X$ is additionally Polish (with metric $d$) and for every $x\in X$ the mapping $\omega\mapsto d(x,A(\omega))$ is measurable.} to be any set-valued function $\rset:\Omega \to \Borel$ such that the graph $\{(\omega, \rset(\omega))\ : \ \omega\in\Omega\} \subset \Omega \times X$ is measurable in the product $\sigma$-algebra $\F\otimes \Borel$ . One is then able to define \emph{rate of escape} from a given random set under the random dynamics of the cocycle.

\begin{definition}\label{def:escape}
  Let ${T}:\mathbb{N}\times\Omega\to\text{End}(X,\Borel, m)$ be a measurable cocycle over $(\Omega, \F,\prob,\base)$  and let $\rset:\Omega \to \Borel$ be a random set. The \emph{random escape rate} with respect to $m$ is the non-negative valued function $E(\rset, \cdot):\Omega \to \mathbb{R}$ given by
	\begin{equation}\label{eq:escape}
		E(\rset, \omega) := -\limsup_{n\to\infty}\frac{1}{n}\log m(\rset^{(n)}(\omega)), \quad \omega\in\Omega,
	\end{equation}
	where
	\begin{equation}\label{eq:rsetn}
	  \rset^{(n)}(\omega) := \bigcap_{i=0}^{n-1}{{T}(i,\omega)}^{-1}\rset(\base^i\omega).
	\end{equation}
\end{definition}

\begin{figure}[hbt]
  \psfrag{t1}{$T_\omega$}
  \psfrag{t2}{$T_{\vartheta\omega}$}
  \psfrag{t3}{$A(\omega)$}
  \psfrag{t4}{$A(\vartheta\omega)$}
  \psfrag{t5}{$A(\vartheta^2\omega)$}
  \begin{center}
    \includegraphics[width=11cm]{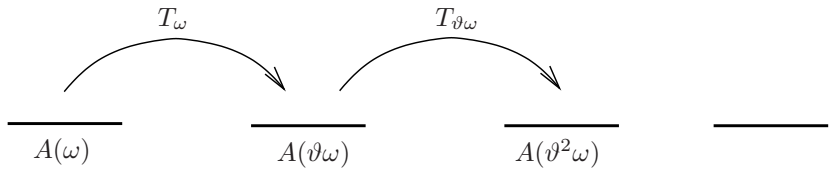}
  \end{center}
 \caption{Schematic of the dynamics between random sets.}
  \label{fig:FLQ09a}
\end{figure}

The random escape rate describes the exponential rate at which trajectories escape from the sequence of sets $A(\omega),A(\vartheta\omega),A(\vartheta^2\omega),\ldots$;  roughly speaking, the Lebesgue measure of points in $A(\omega)$ that remain in this sequence of sets for $n$ iterations is proportional to $\exp(-E(\rset,\omega)n)$.
Clearly, a smaller escape rate $E(\rset,\omega)$ corresponds to a greater proportion of points in $A(\omega)$ having trajectories that remain in the random sequence of sets for a given number of iterations $n$.
One can arguably connect a lower escape rate with a greater ``coherence'' of the sequence $A(\omega),A(\vartheta\omega),A(\vartheta^2\omega),\ldots$.
Returning to the geophysical example in the introduction, if the sets $A(\omega),A(\vartheta\omega),A(\vartheta^2\omega),\ldots$ represent the location of an ocean eddy at weekly intervals, a lower value of $E(\rset,\omega)$ corresponds to an eddy that takes longer to disperse, as a larger proportion of water particles continue to be carried along in the sequence of sets for longer.

In Definition \ref{def:escape} we defined escape rate from a random set. It is often of interest in dynamical systems to study properties of single orbits. In order to study the escape rate along a sample orbit we can restrict the domain of a random set just to this particular orbit.

\begin{definition}
  Let $A:\Omega\to\Borel$ be a random set and let $\omega^*\in\Omega$. We shall refer to the restriction of $A$ to the orbit $\{\base^n\omega^*\}_{n\in\mathbb{Z}^+}$ as an \emph{orbit set}.
\end{definition}

The following proposition shows that any mapping $A:\{\base^n\omega^*\}_{n\in\mathbb{Z}^{+}}\to\Borel$ is an orbit set, as it may be trivially extended to a random set.
\begin{proposition}\label{prop:rsextension}
  For a fixed $\omega^*\in\Omega$, any mapping $A:\{\base^n\omega^*\}_{n\in\mathbb{Z}^+}\to \Borel$ may be extended to a random set by defining $A(\omega) = X$ for all $\omega \in \Omega\setminus\{\base^n\omega^*\}_{n\in\mathbb{Z^+}}$.
\end{proposition}
\begin{proof}
  To see that this extension indeed produces a set $\{(\omega, A(\omega)\}\in\F\otimes\Borel$ note that we may write the graph of $A$ as the union of $(\Omega\setminus\{\base^n\omega^*\}_{n\in\mathbb{Z^+}})\times X$ and $\bigcup_{n\in\mathbb{Z}^+} (\base^n\omega^*, A(\base^n\omega^*))$. The former set is a rectangle in $\F\times\Borel$ and the latter set is a countable union of measurable rectangles, as all singletons are $\F$-measurable. Thus the graph of $A$ is $(\F\otimes\Borel)$-measurable.
\end{proof}

Despite our interest being primarily in orbit sets, we note that when the base is ergodic, under a mild condition a random set has almost-everywhere constant escape rate.

\begin{proposition}\label{prop:er_cae}
  Assume that the base system $(\base, \prob)$ is ergodic and that for almost every $\omega\in\Omega$ the Radon-Nikodym derivative, $\frac{d(m\circ {T}_{\omega}^{-1})}{dm}$, is bounded. For any fixed random set $\rset: \Omega\to\Borel$, $E(\rset, \omega)$ is constant $\prob$-almost everywhere.
\end{proposition}
\begin{proof}
  We begin with the observation that $\rset^{(n)}(\omega) = {T}_\omega^{-1}(\rset^{(n-1)}(\base\omega)) \cap \rset(\omega)$ (see \eqref{eq:rsetn}) so that $m(\rset^{(n)}(\omega)) \le m({T}_{\omega}^{-1}(\rset^{(n-1)}(\base\omega)))$. Using the boundedness of the Radon-Nikodym derivative one sees that $E(\rset, \omega) \le E(\rset, \base\omega)$. We will now show that $E(\rset, \omega) = E(\rset)$ for almost any $\omega\in\Omega$. Assume otherwise; as $\rset$ is a random set it is straightforward to show that $E(\rset, \cdot)$ is measurable and that there exists $c\in\mathbb{R}$ such that the set $S := \{\omega\ :\ E(\rset,\omega) \ge c\}$ has $\prob(S) \in (0,1)$. Since $\base^{-1}(S) \supseteq S$ and $\base$ preserves $\prob$ we must have $\base^{-1}(S) = S$ a.e., but this cannot be since $\prob$ is ergodic.
\end{proof}

On the other hand, supposing that $Z= L^1(X,\Borel, m)$ (or a subspace of $L^1(X,\Borel,m)$),  we may define a \emph{Perron-Frobenius operator cocycle} as follows.

\begin{definition}
  Let ${T}:\mathbb{N}\times \Omega \to \text{End}(X,\Borel, m)$ be a measurable map cocycle over $(\Omega,\F,\prob,\base)$. The corresponding \emph{Perron-Frobenius operator cocycle} is a linear cocycle ${\p}:\mathbb{N}\times \Omega \to \text{End}(L^1(X,\Borel, m))$ whose generator $\tilde\p$ is given by
\begin{equation}\label{eq:pf}
	\int_{B} \tilde\p(\omega) f \ dm = \int_{{T}_\omega^{-1}B}f \ dm, \quad \forall \omega \in \Omega, \ \forall B\in\Borel, \ \forall f\in L^1(X, \Borel, m).
\end{equation}
\end{definition}

\begin{definition}
  Let $\p:\mathbb{N}\times\Omega \to \End(L^1(X,\Borel,m))$ be a Perron-Frobenius operator cocycle corresponding to a measurable map cocycle $T:\mathbb{N}\times \Omega \to \End(X)$. For any $f\in L^1(X,\Borel,m)$, and $\omega\in\Omega$ the \emph{Lyapunov exponent} is defined to be
\begin{displaymath}
	\lambda(\omega, f) := \limsup_{n\to\infty}\frac{1}{n}\log\norm{\p_{\omega}^{(n)}f}_{L^1}.
\end{displaymath}
We also define the \emph{Lyapunov spectrum} to be the set of all Lyapunov exponents: $\Lambda(\omega):= \{\lambda(\omega, f)\ :\ f\in L^{1}(X,\Borel,m)\}$, and the quantity $\lambda(\omega) \in \mathbb{R}$ by $\lambda(\omega):= \limsup_{n\to\infty}\frac{1}{n}\log\norm{\p_{\omega}^{(n)}}_{op}$.
\end{definition}

As each $\p_{\omega}:L^1\circlearrowleft$ is a Markov operator we have $\norm{\p_{\omega}f}_{L^1}\le\norm{f}_{L^1}$ and therefore $\Lambda(\omega)\subseteq [-\infty,0]$. With the definition of the Perron-Frobenius cocycle in \eqref{eq:pf}, it is natural to use the $L^1$-norm for calculating the Lyapunov spectrum. However we will see later in Section \ref{sec:oseledets} that when working with subspaces of $L^1$ other norms are sometimes more informative.

Our main result relates the Lyapunov exponents of a Perron-Frobenius operator cocycle ${\p}$ to the rates of escape from particular orbit sets under the corresponding measurable map cocycle ${T}$.

\begin{theorem}[Main Theorem]\label{thm:main}
  Let ${T}:\mathbb{N}\times\Omega \to \End(X,\Borel, m)$ be a measurable map cocycle over $(\Omega, \F, \prob, \base)$ and let $\p:\mathbb{N}\times\Omega \to \End(L^1(X,\Borel,m))$ be the corresponding Perron-Frobenius cocycle as defined in \eqref{eq:pf}. Fix an aperiodic $\omega^* \in \Omega$ and suppose that there exists an $f\in L^\infty$ such that $\lambda(\omega^*, f) < 0$. Let $\oset_{+},\oset_{-}:\{\base^n\omega^*\}_{n\in\mathbb{Z}^{+}}\to \Borel$ be defined by
  \begin{equation}\label{eq:mssets}
    \oset_{\pm}(\base^n\omega^*):= \{x\in X\ :\ \pm \p_{\omega^*}^{(n)}f(x) > 0\}, \quad n\in\mathbb{Z}^+.
  \end{equation}
  Then $A_{\pm}$ are orbit sets and one has $E(\oset_{\pm},\omega^*)\le -\lambda(\omega^*, f)$.
\end{theorem}

The fact that the sets defined in \eqref{eq:mssets} are orbit follows from Proposition \ref{prop:rsextension}. The proof of the rest of Theorem \ref{thm:main} follows after a preliminary lemma.

\begin{lemma}\label{lem:main}
  In the notation of Theorem \ref{thm:main} we have for every $n\in\mathbb{Z}^{+}$
	\begin{displaymath}
	  \int_{\oset_{+}(\base^n\omega^*)}\p_{\omega^*}^{(n)}f \ dm = \frac{1}{2}\norm{\p_{\omega^*}^{(n)}f}_{L^1}.
	\end{displaymath}
\end{lemma}
\begin{proof}
  Firstly we will show that $\int_{X}\p_{\omega^*}^{(n)}f \ dm = 0$ for all $n\ge 0$. From \eqref{eq:pf} one can see that $\p_{\omega^*}$ preserves integrals over all of $X$, therefore $\int_{X}\p_{\omega^*}^{(n)}f\ dm = M$ -- a constant for all $n\ge 0$. This implies that $\norm{\p_{\omega^*}^{(n)}f}_{L^1} \ge \abs{M}$ for all $n\ge 0$. Suppose that $M\neq 0$. Then
	\begin{displaymath}
          \lambda(\omega^*, f) = \limsup_{n\to\infty}\frac{1}{n} \log \norm{\p_{\omega^*}^{(n)}f}_{L^1} \ge \limsup_{n\to\infty}\frac{1}{n}\log \abs{M} = 0.
	\end{displaymath}
        This is a contradiction as $\lambda(\omega^*, f) < 0$, therefore $M=0=\int_X \p_{\omega^*}^{(n)}f \ dm$. Now we have
\begin{equation}
  0 = \int_{\oset_{+}(\base^n\omega^*)} \p_{\omega^*}^{(n)}f \ dm + \int_{X\setminus\oset_{+}(\base^n\omega^*)}\p_{\omega^*}^{(n)}f \ dm
	\label{eq:lem1}
\end{equation}
and
\begin{equation}
  \norm{\p_{\omega^*}^{(n)}f}_{L^1} = \int_{\oset_{+}(\base^n\omega^*)} \p_{\omega^*}^{(n)}f\ dm - \int_{X\setminus \oset_{+}(\base^n\omega^*)} \p_{\omega^*}^{(n)}f\ dm.
  \label{eq:lem2}
\end{equation}
Adding equations \eqref{eq:lem1} and \eqref{eq:lem2} yields the required result.
\end{proof}
\begin{proof}[Proof of Theorem \ref{thm:main}]
  For any set $S\subseteq X$ denote the indicator function of $S$ by $\mathbbm{1}_{S}:X\to\{0,1\}$. Let $j,n$ be integers such that $0\le j \le n$ and let $B\in\Borel$. Using \eqref{eq:pf} we derive the following:

	\begin{align*}
	  \int_{B} \p_{\omega^*}^{(j+1)}f \ dm &= \int_{T_{\base^{j}\omega^*}^{-1}B}\p_{\omega^*}^{(j)}f \ dm\\
	  & = \int_{T_{\base^{j}\omega^*}^{-1}B} (\p_{\omega^*}^{(j)}f)\mathbbm{1}_{\rset_{+}(\base^{j}\omega^*)} \ dm + \int_{T_{\base^{j}\omega^*}^{-1}B} (\p_{\omega^*}^{(j)}f) \mathbbm{1}_{X\setminus \rset_{+}(\base^{j}\omega^*)}\ dm\\
	  & \le \int_{T_{\base^{j}\omega^*}^{-1}B} (\p_{\omega^*}^{(j)}f)\mathbbm{1}_{\rset_{+}(\base^{j}\omega^*)} \ dm \\
	  &= \int_{T_{\base^{j}\omega^*}^{-1}B \cap \rset_{+}(\base^{j}\omega^*)} \p_{\omega^*}^{(j)}f \ dm.
	\end{align*}
Hence
\begin{displaymath}
  \int_{B} \p_{\omega^*}^{(j+1)}f \ dm \le \int_{T_{\base^{j}\omega^*}^{-1}B\cap \rset_{+}(\base^{j}\omega^*)} \p_{\omega^*}^{(j)}f \ dm.
\end{displaymath}
Now letting $B = \rset_{+}^{(n-j-1)}(\base^{j+1}\omega^*)$ (defined as in \eqref{eq:rsetn}) we have for all $j\ge 0$
\begin{displaymath}
  \int_{\rset_{+}^{(n-j-1)}(\base^{j+1}\omega^*)} \p_{\omega^*}^{(j+1)}f \ dm \le \int_{\rset_{+}^{(n-j)}(\base^{j}\omega^*)} \p_{\omega^*}^{(j)}f \ dm,
\end{displaymath}
where we have used the relation $\rset_{+}^{(n-j)}(\base^{j}\omega^*) = \rset_{+}(\base^{j}\omega^*) \cap T_{\base^{j}\omega^*}^{-1}(\rset_{+}^{(n-j-1)}(\base^{j+1}\omega^*))$, easily obtainable from \eqref{eq:rsetn}. By considering all $j = 0,1,\dots , n-1$ we arrive at the following series of inequalities:
\begin{displaymath}
  \int_{\rset_{+}^{(0)}(\base^n\omega^*)} \p_{\omega^*}^{(n)}f \ dm \le \int_{\rset_{+}^{(1)}(\base^{n-1}\omega^*)} \p_{\omega^*}^{(n-1)}f \ dm \le \cdots \le \int_{\rset_{+}^{(n)}(\omega^*)} f \ dm.
\end{displaymath}
Hence
\begin{displaymath}
  \frac{1}{2} \norm{\p_{\omega^*}^{(n)}f} = \int_{\rset_{+}(\base^n\omega^*)} \p_{\omega^*}^{(n)}f \ dm \le \int_{\rset_{+}^{(n)}(\omega^*)} f \ dm \le \norm{f}_{L^{\infty}} m(\rset_{+}^{(n)}(\omega^*)),
\end{displaymath}
where the equality above is due to Lemma \ref{lem:main}, and the second inequality holds because $f \in L^{\infty}$.
By taking logarithms, dividing by $n$ and taking limit supremum as $n\to\infty$ we arrive at the required inequality $E(\rset_{+}, \omega^*) \le -\lambda(\omega^*, f)$. The inequality for $E(\rset_{-},\omega^*)$ is obtained similarly by considering $-f$ in place of $f$.
\end{proof}

\begin{remark}\label{rem:conditional}
 We note that if for a random set $\rset:\Omega \to \Borel$ (or an orbit set $\oset : \{\base^n\omega^*\}_{n\in\mathbb{Z}^{+}}\to\Borel$) one defines a \emph{conditional operator cocycle} ${\p}_\rset$ by $\tilde{\p}_\rset(\omega)f := \tilde\p(\omega)(f\chi_{\rset(\omega)})$ for all $\omega\in\Omega$ (or $\omega \in \{\base^n\omega^*\}_{n\in\mathbb{Z}^{+}}$) and $f\in L^1$, then Lebesgue escape rate is given by $E(\rset,\omega) = -\limsup_{n\to\infty} (1/n) \log \norm{\p_{\rset}(n,\omega)\mathbbm{1}}_{L^1}$, which is equal in absolute value to the Lyapunov exponent of a constant function with respect to this conditional cocycle.
\end{remark}
\begin{remark}
  Note that the sets $\oset_{\pm}$ defined by \eqref{eq:mssets} are only guaranteed to be orbit sets when $\omega^*$ is aperiodic. If $\omega$ is periodic, one would further require $f$ to be an eigenfunction of $\p^{(p)}_{\omega^*}$ (where $p$ is the period). This situation has been treated by Theorem 2.4 of \cite{FS10}.
\end{remark}

\subsection{Choosing a metastable partition}

Theorem \ref{thm:main} presents a method of finding pairs of orbit sets whose $\omega$-fibres form 2-partitions of $X$. Both of these orbit sets have low escape rates. Theorem \ref{thm:main} applies to a large class of random dynamical systems. For the remainder of this section we will investigate some of the consequences of this result. Lemma \ref{lem:FullSpectrum} will show that in a very general setting one may choose any $\rho\in[-\infty,0)$, find an appropriate $f: \Omega\to L^{1}(X, \Borel, m)$ with Lyapunov exponent $\lambda(\omega, f_{\omega}) = \rho$ for almost every $\omega\in\Omega$, and obtain two random sets with escape lower than $-\rho$. In particular $\rho$ may be arbitrarily close to $0$, however as we will see later choosing such $\rho$ often results in highly irregular metastable random sets.

\begin{definition}
  A mapping $h:\Omega\to L^1(X,\Borel,m)$ is said to be a \emph{random $L^1$-function} if $(\omega, x)\mapsto h(\omega, x)$ is $(\F\otimes \Borel, \Borel(\mathbb{R}))$-measurable. If each $h_{\omega}$ is a density in $L^1(X,\Borel,m)$, it is called a \emph{random density}. Such a density is said to be \emph{preserved} by a Perron-Frobenius operator cocycle $\p$ if $\p_{\omega}h_{\omega} = h_{\base\omega}$ for almost every $\omega\in\Omega$.
\end{definition}

\begin{lemma}\label{lem:FullSpectrum}
  Let $\p:\mathbb{N}\times \Omega \to \End(L^1(X,\Borel, m))$ be a Perron-Frobenius operator cocycle (of a measurable map cocycle $T$) over $(\Omega,\F,\prob,\base)$ that preserves a positive random density $h:\Omega \to L^{1}(X, \Borel, m)$. Suppose that there exists a random function $g:\Omega\to L^{1}(X,\Borel, m)$ so that $\p_{\omega}g_{\omega} = 0$ for almost all $\omega \in \Omega$. Then for every $\rho\in[-\infty,0]$ there exists a random function $f:\Omega \to L^1(X, \Borel, m)$ such that $\lambda(\omega, f_{\omega}) = \rho$ for almost every $\omega\in\Omega$.
\end{lemma}
\begin{proof}
  We modify an argument from Baladi \cite{Bal00} (Theorem 1.5 (7)). Define $f$ so that $f_{\omega} := \sum_{n=0}^{\infty} e^{\rho n}(g_{\base^n\omega}/h_{\base^n\omega})\circ {T}^{(n)}_{\omega} \cdot h_{\omega}$ for every $\omega\in\Omega$. Note that $f$ is also a random $L^1$-function. For any $B\in\Borel$ we have
	\begin{align*}
		\int_{B} \p_{\omega} f_{\omega} \ dm &= \int_{{T}_{\omega}^{-1}B} f_{\omega} \ dm\\
 &=\int_{{T}_{\omega}^{-1}B} \sum_{n=0}^{\infty} e^{\rho n}(g_{\base^n\omega}/h_{\base^n\omega})\circ {T}_{\omega}^{(n)}\cdot h_{\omega} \ dm\\
		&= \int_{{T}_{\omega}^{-1}B} g_{\omega} \ dm + \int_{{T}_{\omega}^{-1}B} \sum_{n=1}^{\infty} e^{\rho n}(g_{\base^n\omega}/h_{\base^n\omega})\circ {T}_{\omega}^{(n)}\cdot h_{\omega} \ dm\\
		&= 0 + e^{\rho} \int_{{T}_{\omega}^{-1}B} \sum_{n=0}^{\infty} e^{\rho n}(g_{\base^{n+1}\omega}/h_{\base^{n+1}\omega}) \circ {T}_{\omega}^{(n+1)}\cdot h_{\omega} \ dm\\
		&= e^{\rho} \int_{B}\sum_{n=0}^{\infty} e^{\rho n} (g_{\base^{n+1}\omega}/h_{\base^{n+1}\omega})\circ {T}_{\base\omega}^{(n)}\cdot h_{\base\omega}\ dm\\
		&=e^{\rho}\int_{B} f_{\base\omega}.
	\end{align*}
	Thus $\p_{\omega} f_{\omega} = e^{\rho}f_{\base\omega}$ almost everywhere. Now for $\epsilon > 0$ let $\Omega_{\epsilon} := \{\omega \in \Omega\ :\ \norm{f_{\omega}}_{L^1} \ge \epsilon \}$. Since $\omega\mapsto\norm{f_{\omega}}_{L^1}$ is measurable, the set $\Omega_{\epsilon}$ is also measurable. Fix $\epsilon$ sufficiently small so that $\prob(\Omega_{\epsilon}) > 0$. The Poincar\'e Recurrence Theorem asserts that $\prob$-almost surely there is a sequence $m_{k}\uparrow\infty$ such that $\base^{m_{k}}\omega \in \Omega_{\epsilon}$. Hence
	\begin{displaymath}
		0 \ge \limsup_{n\to\infty}\frac{1}{n}\log\norm{f_{\base^n\omega}}_{L^1}\ge \limsup_{k\to\infty}\frac{1}{m_{k}}\log\norm{f_{\base^{m_k}\omega}}_{L^1}\ge 0,
	\end{displaymath}
	from which we obtain
	\begin{align*}
		\lambda(\omega, f) = \limsup_{n\to\infty}\frac{1}{n}\log\norm{\p_{\omega}^{(n)}f_{\omega}}_{L^1} &= \limsup_{n\to\infty}\frac{1}{n}\log e^{\rho n}\norm{f_{\base^{n}\omega}}_{L^1}\\
                &= \rho + \limsup_{n\to\infty}\frac{1}{n}\log\norm{f_{\base^n\omega}}_{L^1} = \rho.
	\end{align*}
\end{proof}

It is clear that the set-valued mappings $A_{\pm}:\Omega \to \Borel$ defined by $A_{\pm}(\omega) := \{\pm f_{\omega} > 0\}$ obtained from a random function $f$ are indeed random sets. Thus an application of Theorem \ref{thm:main} to $f$ in Lemma \ref{lem:FullSpectrum} implies that for any negative $\rho$, arbitrarily close to zero, there exist complementary random sets whose rate of escape is slower than $-\rho$.

\begin{example}\label{ex:FullSpectrum}
  Let $(\Omega, \F, \prob, \base)$ be the full two-sided 2-shift on $\{0,1\}$ equipped with the $\sigma$-algebra $\F$ generated by cylinders, and with Bernoulli probability measure $\prob$. Let $\tilde T:\Omega\to\End([0,1])$ be the generator of a cocycle ${T}$, constant on cylinders, given by $\tilde T(\omega) := T_{\omega_0}$ where $T_0(x) := 2x + \alpha_0\mod 1$ and $T_1(x) = 2x + \alpha_1\mod 1$, for some $\alpha_0, \alpha_1\in\mathbb{R}$. It is easy to check that the corresponding Perron-Frobenius operator cocycle ${\p}$ satisfies Lemma \ref{lem:FullSpectrum} with $h_{\omega} \equiv 1$ for all $\omega$ and $g_\omega = g_{\omega_0}$ where
	\begin{displaymath}
		g_i(x) =
		\begin{cases}
			-1/2 & \text{if } 0 \le x-\alpha_i \mod 1\le 1/2\\
			1/2 & \text{if } 1/2  < x-\alpha_i \mod 1 \le 1.
		\end{cases}, \quad i = 0,1.
	\end{displaymath}
        After applying Lemma \ref{lem:FullSpectrum} we conclude that any $\rho\in[-\infty, 0]$ is a Lyapunov exponent hence, by Theorem \ref{thm:main}, there exist complementary random sets with arbitrarily low escape rates.

	For a numerical demonstration  we set $\alpha_0 = 0$ and $\alpha_1 = 0.6$. We choose $\omega^*\in\Omega$ such that $\omega^{*}_{i} = 0$ for all $i<0$ and $\omega^*_i$ equals the $(i+1)^{th}$ digit in the fractional part of the binary expansion of $\pi$ for $i\ge 0$. The first few central elements of $\omega^{*}$, with the zeroth element underlined, are:
	\begin{displaymath}
	  \omega^{*} = (\dots, 0 , 0, \underline{0}, 0,1,0,0,1,0,0,0,0,1,1,1,1,1,\dots).
	\end{displaymath}
	
	Numerical approximations of $f_{\omega^{*}}$ for some values of $\rho$ are shown in Figure \ref{fig:LyapunovFunctions}. For $\rho=-1$ applying the construction in Theorem \ref{thm:main} we see from the graph of $f_{\omega^{*}}$ that $\rset_{-}(\omega^{*}) = [0,1/2)$ and $\rset_{+}(\omega^{*}) = [1/2,1]$. As $\rho$ becomes closer to $0$ we can see more oscillations in $f_{\omega^{*}}$ and subsequently higher disconnectedness of the corresponding sets $\rset_{\pm}(\omega^*)$.
\begin{figure}[hbt]
  \begin{center}
\psfrag{x}{\small $x$}
\psfrag{fomega}{\small $f_{\omega^{*}}(x)$}
\psfrag{rhonegone}{\small $\rho=-1$}
\psfrag{rhonegptfour}{\small $\rho=-0.4$}
\psfrag{rhonegptone}{\small $\rho=-0.1$}
\includegraphics[width=13cm]{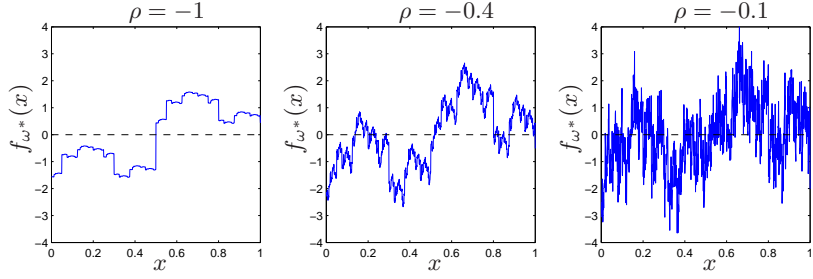}
\caption{Graphs of $f_{\omega^{*}}$ corresponding to different Lyapunov exponents in Example \ref{ex:FullSpectrum}. Note the increased irregularity as $\rho$ approaches zero.}
\label{fig:LyapunovFunctions}
  \end{center}
\end{figure}
\end{example}

\section{Oseledets splittings and applications}\label{sec:oseledets}
In this section we extend Theorem \ref{thm:main} to apply in a Banach space $(Y, \norm{\cdot}_{Y})$, with $Y\subset L^1(X)$, in which the Perron-Frobenius cocycle admits an \emph{Oseledets splitting}. We then apply these new results to expanding maps of the unit interval, $I=[0,1]$, where $Y$ is taken to be the space of functions of bounded variation $\bv$.

\begin{definition}[\cite{Thi87, FLQ10}]
  A linear operator cocycle ${\p}:\mathbb{N}\times\Omega \to \End(Y,\norm{\cdot}_{Y})$ over $(\Omega,\F,\prob,\base)$ is said to be \emph{quasi-compact} if for almost every $\omega$ there exists an $\alpha<\lambda(\omega)$ such that the set $\mathcal{V}_{\alpha} := \{y\in Y\ :\ \lambda(\omega, y) < \alpha\}$ is finite co-dimensional. We will denote the infimal such $\alpha$ by $\alpha(\omega)$.
\end{definition}

Quasi-compact cocycles have the property that Lyapunov exponents larger than $\alpha(\omega)$ are isolated. For an isolated Lyapunov exponent $r>\alpha(\omega)$, let $\epsilon > 0$ be small enough so that $\Lambda(\omega) \cap (r-\epsilon,r) = \emptyset$. If the co-dimension of $\mathcal{V}_{r-\epsilon}(\omega)$ in $\mathcal{V}_{r}(\omega)$ is $d$ then we call $r$ a Lyapunov exponent of \emph{multiplicity} $d$. There are at most countably many of these and we refer to them as \emph{exceptional Lyapunov exponents}. The \emph{exceptional Lyapunov spectrum} is the set of pairs of exceptional Lyapunov exponents and their multiplicities, $\{ (\lambda_i(\omega), d_i(\omega))\}_{i=1}^{p(\omega)}$. From now on, we retain the assumption that the base system $(\base,\prob)$ is ergodic, which ensures that $\lambda_i$, $d_i$ and $p$ are all constant almost everywhere.

By $\mathcal{G}_{d}(Y)$ and $\mathcal{G}^{c}(Y)$ we will denote the subspaces of the Grassmannian $\mathcal{G}(Y)$ of $Y$ consisting only of subspaces of dimension $d$ and codimension $c$ respectively;  see the Appendix for more details on Grassmannians and their topology.

\begin{definition}[Oseledets splitting \cite{Thi87}]
  A quasi-compact linear operator cocycle $\p:\mathbb{N}\times\Omega\to\End(Y,\norm{\cdot}_{Y})$ over $(\Omega,\F,\prob,\base)$ with exceptional spectrum $\{(\lambda_i, d_i)\}_{i=1}^{p}$, $p\le\infty$, admits a \emph{Lyapunov filtration} over a $\base$-invariant set $\tilde\Omega\subseteq \Omega$ of full measure, if there exists a collection of maps $\{V_i : \Omega \to \mathcal{G}^{c_i}(Y)\}_{i=1}^{p}$, such that for all $\omega\in\tilde\Omega$ and all $i=1,\dots,p$
	\begin{itemize}
		\item[(i)] $Y = V_{1}(\omega) \supset \cdots \supset V_{i}(\omega) \supset V_{i+1}(\omega)$
		\item [(ii)] $\mathcal{V}_{\alpha(\omega)} \subseteq \cap_{i=1}^{p}V_i(\omega)$, with equality if and only if $p$ is infinite;	 
		\item [(iii)] $\p_{\omega} V_{i}(\omega) = V_{i}(\base\omega)$;
		\item [(iv)] $\lambda(\omega, v) = \lim_{n\to\infty}\frac{1}{n} \log\norm{\p_{\omega}^{(n)}v}_{Y} = \lambda_i$ if and only if $v\in V_{i}(\omega)\setminus V_{i+1}(\omega)$. If $p$ is finite, we take $V_{p+1}(\omega) := \mathcal{V}_{\alpha(\omega)}(\omega)$.
	\end{itemize}
        An \emph{Oseledets splitting} for $\p$ is a Lyapunov filtration with an additional family of maps $\{E_{i}: \Omega \to \mathcal{G}_{d_i}(Y)\}_{i=1}^{p}$ such that for all $\omega\in\tilde\Omega$ and $i=1,\dots,p$
	\begin{itemize}
	  \item [(v)] $V_{i}(\omega) = E_{i}(\omega)\oplus V_{i+1}(\omega)$ (with $V_{p+1}(\omega) := \mathcal{V}_{\alpha(\omega)}(\omega)$ for $p<\infty$);
		\item [(vi)] $\p_{\omega}E_{i}(\omega) = E_{i}(\base\omega)$;
		\item [(vii)] $\lambda(\omega, v) = \lambda_i$ if $v\in E_{i}(\omega)\setminus \{0\}$.
	\end{itemize}
\end{definition}

A Lyapunov filtration is measurable if each $V_i:\Omega \to \mathcal{G}^{c_i}(Y)$ is $(\F, \Borel(\mathcal{G}^{c_i}(Y)))$-measurable. An Oseledets splitting is measurable if its Lyapunov filtration is measurable and each of the maps $E_i:\Omega\to \mathcal{G}_{d_i}(Y)$ is measurable. For more details on measurability we refer the reader to the appendix.

In order to connect the $Y$-Lyapunov spectrum to escape rates, we first need to relate the $Y$-Lyapunov exponents to the $L^1$-Lyapunov exponents used in Theorem \ref{thm:main}. For this we shall require a certain relation between the two norms.

\begin{theorem}\label{thm:equiv_norms}
  Let ${\p}: \mathbb{N}\times\Omega \to \End(Y,\norm{\cdot}_Y)$ be a quasi-compact linear operator cocycle over $(\Omega, \mathcal{F}, \prob, \base)$ with exceptional spectrum $\{(\lambda_i, d_i)\}_{i=1}^{p}$ and a measurable Oseledets splitting $\{E_i\}_{i=1}^{p}$ on $\tilde\Omega$. Let $\norm{\cdot}_{*}$ be a second norm on $Y$ such that $\norm{\cdot}_{*} \le C\norm{\cdot}_{Y}$ for some $C>0$. Then for almost any $\omega\in\tilde\Omega$, $i\in\{1,\dots,p\}$ and any $f\in E_{i}(\omega)$, we have $\lambda_{\norm{\cdot}_{*}}(\omega, f) = \lambda_{\norm{\cdot}_{Y}}(\omega, f) = \lambda_i$; that is, the Lyapunov exponents with respect to the two norms are equal almost everywhere.
\end{theorem}
\begin{proof}
  Firstly note that scaling a norm by a constant factor does not change the Lyapunov exponent, hence without loss of generality we may assume that $C=1$. Fix $i\in \{1,\dots,p\}$. Since $\norm{\cdot}_{*}\le \norm{\cdot}_{Y}$ the inequality $\lambda_{\norm{\cdot}_{*}}(\omega, f) \le \lambda_{\norm{\cdot}_{Y}}(\omega,f)$ for all $\omega \in \tilde\Omega$ follows trivially. Now for the reverse inequality: define a function $c:\tilde\Omega\to\R$ by $c(\omega) = \sup_{\xi\in E_i(\omega)}\norm{\xi}_{Y}/\norm{\xi}_{*} = \psi\circ E_i(\omega)$ where $\psi:\mathcal{G}_{d_i} \to \R$ is as in Lemma \ref{lem:usc}. Since $E_i$ is $(\F, \Borel(\mathcal{G}_{d_i}(X)))$-measurable and $\psi$ is $(\Borel(\mathcal{G}_{d_i}(X)), \Borel(\R))$-measurable, their composition $c$ is $(\F, \Borel(\R))$-measurable.

  For a positive integer $N$ let $\mathbbm{1}_{ \{c<N\}}$ be the indicator function of the (measurable) set $\{\omega\ :\ c(\omega) < N\}$. Given any $\omega\in\tilde\Omega$, the function $c(\omega)$ is finite so $\mathbbm{1}_{\{c<N\}}(\omega) = 1$ for all $N > c(\omega)$. Thus $\mathbbm{1}_{\{c<N\}} \to 1$ pointwise. By Lebesgue's Dominated Convergence Theorem, we see that $\prob(\{c<N\}) \to 1$ as $N\to\infty$. Thus we may choose an $N$ large enough so that $\prob(\{c<N\}) > 0$. By Poincar\'e recurrence there almost surely exists a sequence $m_k\uparrow \infty$ such that $\base^{m_k}\omega \in \{c<N\}$. Then
	\begin{displaymath}
	  \lambda_{\norm{\cdot}_{*}}(\omega, f) \ge \limsup_{k\to\infty}\frac{1}{m_k}\log\norm{\p_{\omega}^{(m_k)}f}_{*} \ge \lim_{k\to\infty}\frac{1}{m_k}\log N^{-1}\norm{\p_{\omega}^{(m_k)}f}_{Y} = \lambda_i(\omega),
	\end{displaymath}
which completes the proof.
\end{proof}

\begin{remark}
  By reversing the appropriate inequalities in the proof of Theorem \ref{thm:equiv_norms} and a similar modification of Lemma \ref{lem:usc} one can see that the same result holds when the two norms satisfy the relation $C\norm{\cdot}_{*} \ge \norm{\cdot}_{Y}$ for some $C>0$. In particular Theorem \ref{thm:equiv_norms} is satisfied when the two norms are equivalent.
\end{remark}

A direct consequence of Theorem \ref{thm:main} and Theorem \ref{thm:equiv_norms} is:
\begin{corollary}\label{cor:rychlik}
  Let ${T}:\mathbb{N}\times\Omega \to \End(X,\Borel, m)$ be a measurable map cocycle over $(\Omega, \F, \prob, \base)$ and let its Perron-Frobenius cocycle be ${\p}: \mathbb{N}\times\Omega \to \End(Y,\norm{\cdot}_{Y})$, where $Y\subseteq L^1(X)$ and $\norm{\cdot}_{L^1}\le \norm{\cdot}_{Y}$. Suppose that ${\p}$ is quasi-compact, with exceptional spectrum $\{(\lambda_i,d_i)\}_{i=1}^{p}$, and admits a measurable Oseledets splitting $E_i:\Omega\to\mathcal{G}(X)$. For $\prob$-almost all $\omega^*\in\Omega$ and any $f\in E_i(\omega^*)$ the orbit sets given by $\oset_{\pm}(\vartheta^n\omega^*) = \{\pm \p_{\omega^{*}}^{(n)}f > 0\}$ satisfy $E(\oset_{\pm}, \omega^*) \le -\lambda_i$, $i=2,\dots,p$.
\end{corollary}

This result extends the application of Theorem \ref{thm:main} to Perron-Frobenius cocycles on Banach spaces for which the cocycle is quasi-compact and the Banach space norm dominates the $L^1$-norm. Note that our result also applies to periodic $\omega^{*}$ as, in this case, the corresponding Oseledets subspaces $E_{i}(\omega^{*})$ would indeed be eigenspaces.

\subsection{Application to cocycles of expanding interval maps}

We now focus on the unit interval, $I=[0,1]$, one-dimensional map cocycles $T:\mathbb{N}\times \Omega \to \End(I)$, and their Perron-Frobenius operators. In \cite{FLQ10} it is shown that the Perron-Frobenius cocycle is quasi-compact if the \emph{index of compactness} (a quantity corresponding to the essential spectral radius in the deterministic setting) $\kappa:=\lim_{n\to\infty}(1/n)\log(1/\essinf((T_{\omega}^{n})'(x)))<0$.  This formula for $\kappa$ suggests that any Lyapunov spectral points lying between $\kappa$ and $0$ (the latter corresponding to the random invariant density) are associated with large-scale structures responsible for rates of mixing slower than the local expansion of trajectories can account for. These structures are commonly referred to as \emph{coherent structures, coherent sets} or \emph{(random) almost invariant sets} \cite{FLQ09, FLS10}. We apply the results of Corollary \ref{cor:rychlik} to show that these sets also posses a slow rate of escape, bounded by the corresponding exponent in the Lyapunov spectrum.

Let $(I, \Borel, \ell)$ be the unit interval $[0,1]$ with Borel $\sigma$-algebra and Lebesgue measure. Recall that \emph{variation} of a function $f\in L^1(I)$ is given by
\begin{displaymath}
	\var(f) := \inf_{g\sim f}\sup \sum_{i=1}^{k}\abs{g(p_i) - g(p_{i-1})},
\end{displaymath}
where $g\sim  f$ means that $g$ is an $L^1$ version of $f$ and the supremum is taken over all finite sets $\{p_1<p_2<\cdots <p_k\}\subset I$. Define $\bv := \{f\in L^1(I)\ :\ \var(f) < \infty\}$ to be the space of functions of bounded variation, equipped with the norm $\norm{\cdot}_{\bv} := \max\{\norm{\cdot}_{L^1}, \var(\cdot) \}$.

Let $\Omega \subseteq \{1,\dots, k\}^{\mathbb{Z}}$ be a shift space on $k$ symbols with the left shift map $\base:\Omega\circlearrowleft$ given by $(\base\omega)_j = \omega_{j+1}$. Furthermore, suppose $\F$ is the Borel $\sigma$-algebra generated by cylinders in $\Omega$ and suppose that $\prob$ is an ergodic shift-invariant probability measure on $\Omega$.

A \emph{Rychlik map cocycle} is a cocycle $T:\mathbb{N}\times\Omega \to \End(I)$ obtained from a collection of $k$ \emph{Rychlik maps} $\{T_i\}_{i=1}^{k}$ where the generator $\tilde T$ is given by $\tilde T_{\omega} = T_{\omega_0}$. Rychlik maps \cite{Ryc83}, a generalisation of Lasota-Yorke maps, form a large class of almost-everywhere $C^{1}$ maps of the unit interval whose reciprocal of the modulus of the derivative has bounded variation. We will denote the corresponding Perron-Frobenius operator cocycle $\p : \mathbb{N}\times\Omega \to \End(\bv)$. For more details we refer the reader to \cite{FLQ10}.

In Corollary 28 \cite{FLQ10} it is shown that the Perron-Frobenius cocycle of any Rychlik map cocycle that is \emph{expanding-on-average} (that is, $\kappa <0$) admits a $\prob$-continuous (and therefore measurable) Oseledets splitting in $\bv$. We combine this result with Corollary \ref{cor:rychlik} to obtain the following.


\begin{corollary}\label{cor:rychlik2}
  Let $T:\mathbb{N}\times\Omega\to\End(I)$ be a Rychlik map cocycle which is expanding on average and let $\p:\mathbb{N}\times\Omega\to\End(BV)$ be its Perron-Frobenius operator cocycle, which admits a measurable Oseledets splitting on a set of full $\prob$-measure $\tilde\Omega\subseteq \Omega$. For any isolated Lyapunov exponent $\lambda_i<0$ and $\prob$-almost any $\omega^*\in\Omega$ there exist orbit sets $\rset_{\pm}$ such that $\omega$-fibres of $\rset_{\pm}$ partition $I$ and $E(\rset_{\pm},\omega^*) \le -\lambda_i$.
\end{corollary}

\begin{proof}
  Since $\norm{\cdot}_{L^1} \le \norm{\cdot}_{\bv}$, a direct application of Corollary \ref{cor:rychlik}  shows that any pair of orbit sets $\rset_{\pm}$ satisfying $\rset_{\pm}(\base^n\omega^*) = \{\pm\p_{\omega^*}^{(n)} f > 0\}$ have escape rates lower than $-\lambda_i$.
\end{proof}

Moreover, by an application of Proposition 2.1 \cite{LY78} to $\bv$ functions we see that each $\rset_{\pm}(\omega)$, $\omega \in \{\base^n\omega^{*}\}_{n\in\mathbb{Z}^{+}}$, may be written as a \emph{countable union of closed sets} (including possibly singleton sets). Thus, the orbit sets $\rset_{\pm}(\omega)$, from which we are bounding the rate of escape, have a relatively simple topological form.

\begin{example}\label{ex:rychlik}
  This example is borrowed from \cite{FLQ09} (p746) and we refer the reader to the original article for additional details. It is easy to check that the cocycle $T$ described below is Rychlik and expanding-on-average.
    The base dynamical system is given by a shift $\base$ on sequence space $\Omega = \{\omega \in \{1,\dots, 6\}^{\mathbb{Z}}\ :\ \forall k\in\mathbb{Z}, E_{\omega_k\omega_{k+1}} = 1 \}$
     with transition matrix
    \begin{displaymath}
     E = \left[
     \begin{array}{rrrrrr}
       0 & 1 & 0 & 0 & 1 & 0\\
       0 & 0 & 1 & 0 & 0 & 1\\
       1 & 0 & 0 & 1 & 0 & 0\\
       0 & 0 & 1 & 0 & 0 & 1\\
       1 & 0 & 0 & 1 & 0 & 0\\
       0 & 1 & 0 & 0 & 1 & 0
     \end{array}
     \right],
    \end{displaymath}
    equipped with the $\sigma$-algebra generated by $1$-cylinders and the Markov probability measure $\prob$ determined by the stochastic matrix $\frac{1}{2}E$.

    The map cocycle ${T}$ is generated by maps $\tilde T:\Omega \to \End(I, \Borel, \ell)$ given by $\tilde T(\omega) = T_{\omega_0}$ where $\{T_i\}_{i=1}^{6}$ is a collection of six Lebesgue-preserving, piecewise affine, Markov expanding maps of the interval, which share a common Markov partition (see Figure \ref{fig:FLQ09}).
\begin{figure}[hbt]
  \psfrag{T1}{\small$T_1$}
  \psfrag{T2}{\small$T_2$}
  \psfrag{T3}{\small$T_3$}
  \psfrag{T4}{\small$T_4$}
  \psfrag{T5}{\small$T_5$}
  \psfrag{T6}{\small$T_6$}
  \begin{center}
    \includegraphics[width=11cm]{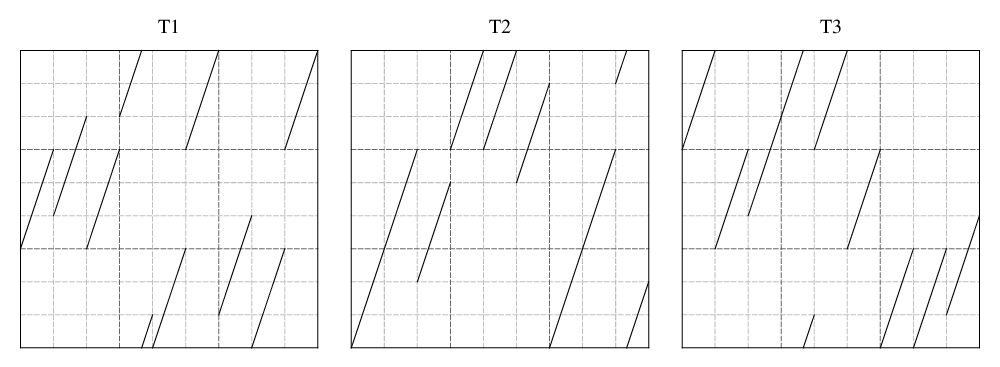}
    \includegraphics[width=11cm]{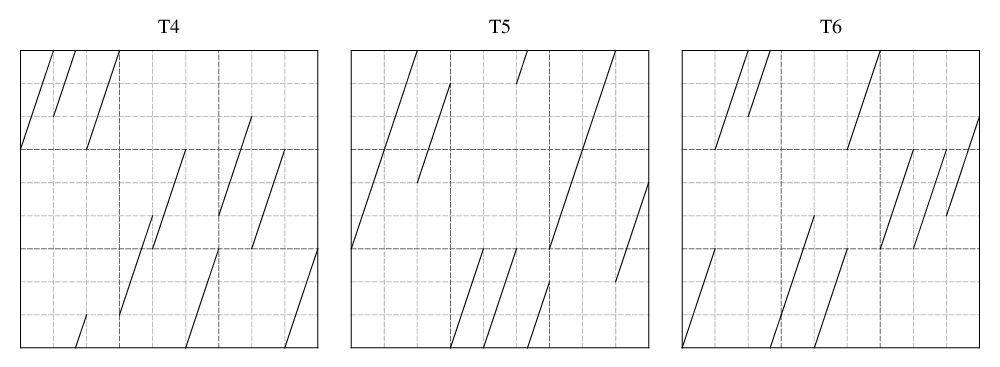}
  \end{center}
 \caption{Graphs of maps $T_{1}, \dots, T_{6}$, reproduced from \cite{FLQ09} (Figure 1).}
  \label{fig:FLQ09}
\end{figure}
The map cocycle $T$ has been designed so that at each step, a particular (random) interval of length 1/3 (selected from $[0,1/3],[1/3,2/3]$ and $[2/3,1]$) is approximately shuffled (with some escape) to another of these three intervals.
For example, the map $T_1$ approximately shuffles $[0,1/3]$ to $[1/3,2/3]$. These particular random intervals are the metastable sets or coherent sets for this random system from which we show the escape rate is slow.

\begin{figure}[hbt]
  \psfrag{fomega}{\small$f_{\omega^{*}}$}
  \psfrag{fomega1}{\small$f_{\base\omega^{*}}$}
  \psfrag{fomega2}{\small$f_{\base^2\omega^{*}}$}
  \psfrag{fomega3}{\small$f_{\base^3\omega^{*}}$}
  \psfrag{fomega4}{\small$f_{\base^4\omega^{*}}$}
  \psfrag{fomega5}{\small$f_{\base^5\omega^{*}}$}
  \psfrag{fomega6}{\small$f_{\base^6\omega^{*}}$}
  \psfrag{fomega7}{\small$f_{\base^7\omega^{*}}$}
  \begin{center}
    \includegraphics[width=13cm]{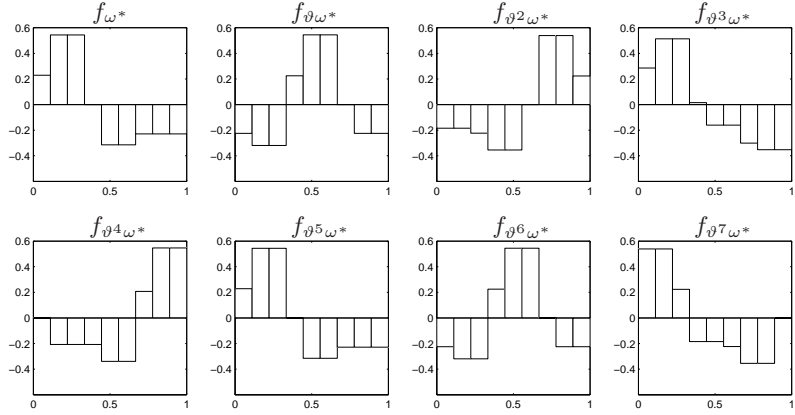}
  \end{center}
\caption{Functions $f_{\base^i\omega^*}$, spanning second Oseledets subspaces $E_{2}(\base^i\omega^*)$ for $i = 0,\dots,7$.}
  \label{fig:RychlikSubspaces}
\end{figure}

A test sequence $\omega^*\in\Omega$ is obtained in the following way. Let $\alpha \in \{0,1\}^{\mathbb{Z}}$ be such that $\alpha_{0} = 0$, $\alpha_i$ is the $(2i)^{th}$ digit in the binary expansion of the fractional part of $\pi$ while $\alpha_{-i}$ is the $(2i-1)^{th}$ digit of the same expansion, $i\ge 1$. Let $h:\Omega \to \{0,1\}^{\mathbb{Z}}$ be such that
\begin{displaymath}
  h(\omega)_i =
  \begin{cases}
    0 & \textnormal{ if } \omega_i \in\{1,2,3\}\\
    1 & \textnormal{ if } \omega_i \in\{4,5,6\}.
  \end{cases}
\end{displaymath}

Observe that $h$ is three-to-one and that we may uniquely choose $\omega^* \in h^{-1}\{\alpha\}$ that satisfies $\omega^*_0 = 1$. Shown below are some of the central elements of $\omega^{*}$, with the zeroth element underlined:

\begin{displaymath}
\omega^{*} = (\dots,3,4,6,5, 4,3,4,6,5,1,2,3,4,3,\underline{1},2,3,4,3,1,5,4,6,5,1,2,3,1,2,\dots).
\end{displaymath}

It is shown in \cite{FLQ09} that $\Lambda(\omega^*) \subset [-\infty, \log1/3] \cup \{\lambda_2(\omega^*)\} \cup \{0\}$ where $\lambda_2(\omega^*)\approx \log0.81$, approximated using the algorithm on p745 \cite{FLQ09}. The functions $f_{\base^n\omega^*} = \p_{\omega^*}^{(n)}f_{\omega^{*}}$ spanning the corresponding Oseledets subspaces $E_{2}(\base^n\omega^*)$ are shown in Figure \ref{fig:RychlikSubspaces}. One can see that, when compared to those in Figure \ref{fig:LyapunovFunctions}, these functions are more regular (i.e. lower variation). We also determine the random metastable sets or coherent sets $\rset_{\pm}(\base^n\omega^*) = \{\pm f_{\base^n\omega^*} > 0\}$ for the first eight values on the forward orbit of $\omega^*$:

\begin{displaymath}
  \begin{array}{rlcrl}
    \rset_{+}(\omega^{*}) &= \left[ 0, 3/9 \right],& \quad & \rset_{-}(\omega^{*}) &= \left[ 3/9, 1 \right),\\
    \rset_{+}(\base\omega^{*}) &= \left[3/9, 6/9 \right],& \quad & \rset_{-}(\base\omega^{*}) &= \left[0, 3/9 \right)\cup \left( 6/9,1 \right],\\
    \rset_{+}(\base^2\omega^{*}) &= \left[ 6/9, 1 \right],& \quad & \rset_{-}(\base^2\omega^{*}) &= \left[0, 6/9 \right),\\
    \rset_{+}(\base^3\omega^{*}) &= \left[0, 4/9 \right],& \quad & \rset_{-}(\base^3\omega^{*}) &= \left(4/9, 1 \right],\\
    \rset_{+}(\base^4\omega^{*}) &= \left[ 6/9, 1 \right],& \quad & \rset_{-}(\base^4\omega^{*}) &= \left[0, 6/9 \right),\\
    \rset_{+}(\base^5\omega^{*}) &= \left[ 0, 3/9 \right],& \quad & \rset_{-}(\base^5\omega^{*}) &= \left[ 3/9, 1 \right),\\
    \rset_{+}(\base^6\omega^{*}) &= \left[3/9, 6/9 \right],& \quad & \rset_{-}(\base^6\omega^{*}) &= \left[0, 3/9 \right)\cup \left( 6/9,1 \right],\\
    \rset_{+}(\base^7\omega^{*}) &= \left[ 0, 3/9 \right],& \quad & \rset_{-}(\base^7\omega^{*}) &= \left[ 3/9, 1 \right).
  \end{array}
\end{displaymath}
As per the discussion in Remark \ref{rem:conditional} we can approximate the rates of escape from $\rset_{+}$ and $\rset_{-}$ by computing the largest Lyapunov exponent of the matrix approximations of the corresponding conditional cocycle (we use $N=0$ and $M=20$ for parameters $M$ and $N$ in the algorithm on p745 \cite{FLQ09}). We then find that $E(\rset_{+},\omega^*)\approx -\log 0.83$ and $E(\rset_{-},\omega^*) \approx -\log0.89$. This is in agreement with Corollary \ref{cor:rychlik2} as both escape rates are less than the previously computed $-\lambda_2(\omega^*)\approx-\log0.81$.

By inspecting $T_{\base^k\omega^*}$ we see that $\rset_{+}(\base^k\omega^{*})$ is mostly mapped onto $\rset_{+}(\base^{k+1}\omega^{*})$, $k=0,\ldots,6$.
\emph{This phenomenon is the cause of the slow escape from the random set $A^+$}.
By Corollary \ref{cor:rychlik2}, \emph{the presence of a Lyapunov spectral value close to 0 forces the existence of orbit sets with escape rates slower than that spectral value.}

\end{example}

\section{Random shifts of finite type and bounds on topological entropy}\label{sec:rsft}

In this section we use our machinery to obtain some results on random shifts of finite type that exhibit metastability, extending some of the results of \cite{FS10} to the random shift setting. For a more detailed description of random shifts of finite type see for example \cite{BG95}. We begin by defining random transition matrices, the corresponding random shifts of finite type and some important properties such as aperiodicity. We alter some of our notation to match the notation usually applied to shifts.
Throughout we assume $(\Omega, \F, \prob, \base)$ is an abstract ergodic base dynamical system such that singletons of $\Omega$ are $\F$-measurable.
\begin{definition}
  For any integer $k\ge 2$, a \emph{random transition matrix} is defined to be a measurable $k\times k$ transition-matrix-valued function $\rmatrix: \Omega \to \mathcal{M}_{k\times k}(\{0,1\})$. For $\omega\in\Omega$ and $n\in\mathbb{N}$ write the \emph{matrix cocycle} as $\rmatrix^{(n)}(\omega) := \rmatrix(\omega)\rmatrix(\base\omega)\cdots \rmatrix(\base^{n-1}\omega)$.
\end{definition}
\begin{definition}
  Let $\mathcal{A} = \{1,\dots,k\}$ be an alphabet and $\mathcal{A}^{\mathbb{Z^{+}}}$ be the space of all one-sided $\mathcal{A}$-valued sequences. A random matrix $M: \Omega\to\mathcal{M}_{k\times k}$ defines a subset of $\mathcal{A}^{\mathbb{Z}^{+}}$ for each $\omega\in\Omega$ by
  \begin{displaymath}
    \Sigma_{\rmatrix}(\omega) := \left\{ x\in\mathcal{A}^{\mathbb{Z}^+} \ : \ \rmatrix_{x_ix_{i+1}}(\base^i\omega) = 1 \text{ for all } i\in \mathbb{Z}^{+} \right\}.
  \end{displaymath}
  Let $\sigma$ be the left shift map on $\mathcal{A}^{\mathbb{Z}^{+}}$. Then $\tau:\{(\omega, \Sigma_{\rmatrix}(\omega))~:~\omega\in\Omega\}\circlearrowleft$ defined by $\tau(\omega, x) := (\base\omega, \sigma x)$ is a skew-product. We shall refer to the bundle random dynamical system determined by the family of mappings $\{\sigma : \Sigma_{\rmatrix}(\omega) \to \Sigma_{\rmatrix}(\base\omega), \ \omega\in\Omega\}$ as a \emph{random shift of finite type} and the $\tau$-invariant set $\Sigma_{\rmatrix}:= \{(\omega,\Sigma_{\rmatrix}(\omega)), \omega\in\Omega\}$ as a \emph{random shift space}.
\end{definition}

\begin{definition}
  A random transition matrix $\rmatrix:\Omega\to\mathcal{M}_{k\times k}(\{0,1\})$ is \emph{aperiodic} if for almost every $\omega\in\Omega$ there exists $N=N(\omega)\in\mathbb{N}$ such that $\rmatrix^{(N)}(\omega)>0$. If $N$ is independent of $\omega$ then $\rmatrix$ is said to be \emph{uniformly aperiodic}. We also use the terms `aperiodic' and `uniformly aperiodic' to describe the corresponding random shift space $\Sigma_{\rmatrix}$.
\end{definition}

\begin{proposition}
	Let $\blocks_n(\omega) = \{[x_0x_1\dots x_{n-1}] \ : \ \rmatrix_{x_i x_{i+1}}(\base^i\omega) = 1 \ \text{for all } 0\le i < n-2\}$ be the set of all $n$-cylinders of $\Sigma_\rmatrix(\omega)$ beginning at position $0$. The following limit exists and is constant $\prob$-almost everywhere:
	\begin{displaymath}
		h(\Sigma_\rmatrix(\omega)) := \lim_{n\to\infty}  \frac{1}{n} \log \abs{\blocks_n(\omega)}.
	\end{displaymath}
\end{proposition}

\begin{proof}
  Observe that $\abs{\blocks_{n+m}(\omega)} \le \abs{\blocks_n(\omega)} \cdot \abs{\blocks_m(\base^n\omega)}$, thus the sequence $\{\log\abs{\blocks_n(\omega)}\}_{n\in\mathbb{Z}^+}$ is subadditive. From Proposition \ref{prop:blocks} that follows it is easy to see that $|\mathcal{C}_n|$ are measurable. By Kingman's Subadditive Ergodic Theorem (see e.g. \cite{Arn98}) there exists a measurable function $f:\Omega\to \mathbb{R}\cup\{-\infty\}$ such that $\lim_{n}(1/n)\log\abs{\blocks_n(\omega)} = f(\omega)$ and $f\circ \base = f$ almost everywhere. As $(\base, \prob)$ is ergodic, $f$ is constant almost everywhere.
\end{proof}

The quantity $h(\Sigma_\rmatrix(\omega))$ is called the \emph{topological entropy} of $\Sigma_\rmatrix(\omega)$. Denote by $h(\Sigma_\rmatrix)$ the constant where $h(\Sigma_\rmatrix) = h(\Sigma_\rmatrix(\omega))$ almost everywhere.

\begin{proposition}
	$\abs{\blocks_n(\omega)} = \sum_{i,j}\rmatrix_{ij}^{(n-1)}(\omega)$ for every $n\ge 2$.
	\label{prop:blocks}
\end{proposition}
\begin{proof}
The proof of this result is largely identical to the proof of its deterministic analogue (see for example Proposition 2.2.12 in \cite{LM95}).
\end{proof}

\begin{definition}
	Let $\{\sigma : \Sigma_\rmatrix(\omega) \to \Sigma_\rmatrix(\base\omega) \}$ and $\{ \sigma : \Sigma_\submatrix(\omega) \to \Sigma_\submatrix(\base\omega)\}$ be two random shifts of finite type with common base dynamical system $(\Omega, \mathcal{F}, \prob, \base)$. $\Sigma_\submatrix$ is a \emph{subshift} of $\Sigma_\rmatrix$ if
	\begin{displaymath}
	  (\submatrix_{ij}(\omega) = 1) \Rightarrow (\rmatrix_{ij}(\omega) = 1) \ \text{ for all } i,j \in \mathcal{A},\  \omega \in \Omega.
	\end{displaymath}
\end{definition}

A subshift may not utilise all the symbols of its parent shift for different values of $\omega$. We may think of this as either a subshift whose alphabet, while finite, changes with $\omega$ or as a subshift on all of the alphabet of its parent shift, but possibly containing isolated vertices in the associated adjacency graph.
We now introduce the notion of a \emph{complementary subshift}.

\begin{definition}
  Let $\Sigma_{\rmatrix}$ be a random shift of finite type and let $\Sigma_\submatrix$ be a subshift of $\Sigma_{\rmatrix}$. The \emph{complementary} subshift of $\Sigma_\rmatrix$ to $\Sigma_{\submatrix}$ is the subshift $\Sigma_{\submatrix'}$ whose elements $\submatrix_{ij}'=1$ if and only if $\rmatrix_{ij}=1$ and $\submatrix_{in}=\submatrix_{nj}=0$ for all $n\in\mathcal{A}$.
\end{definition}

We state a recent extension of the classical Oseledets Multiplicative Ergodic Theorem, which guarantees the existence of an Oseledets splitting of $\mathbb{R}^k$ even when the adjacency matrices $\rmatrix(\omega)$ are not invertible. This is the case in many interesting applications, including shifts. The MET is a central piece of machinery which we use to obtain bounds on topological entropy for certain complementary subshifts.
Later we will see that the leading Lyapunov exponent $\lambda_1$ determines the topological entropy of the shift, while the second Lyapunov exponent $\lambda_2$, if close to $\lambda_1$, indicates the presence of metastability and the possibility of forming complementary subshifts with large entropies relative to that of the original shift.

\begin{theorem}\cite[Theorem 4.1, specialized to adjacency matrices under left multiplication]{FLQ09}
	Suppose $(\Omega, \mathcal{F}, \prob, \base)$ is an ergodic base dynamical system and consider a random transition matrix $\rmatrix : \Omega \to \mathcal{M}_{k\times k}(\{0,1\})$. There exists a forward $\base$-invariant full $\prob$-measure subset $\tilde{\Omega} \subset \Omega$, numbers $\lambda_r < \dots < \lambda_2 < \lambda_1$ and dimensions $d_1, \dots, d_r \in \mathbb{N}$ satisfying $\sum_{\ell}d_\ell = k$ such that for all $\omega\in\tilde{\Omega}$:
	\begin{itemize}
		\item[(i)] There exist subspaces $W_\ell(\omega) \subset \mathbb{R}^{k}$, $\ell=1,\dots,r$, $\textnormal{dim}(W_\ell(\omega)) = d_\ell$;
		\item[(ii)] $\mathbb{R}^k = W_1(\omega)\oplus \dots \oplus W_r(\omega)$ for $\omega\in\tilde\Omega$;
		\item[(iii)] $W_\ell(\omega)\rmatrix(\omega) \subseteq W_\ell(\base\omega)$ with equality if $\lambda_\ell > -\infty$;
		\item[(iv)] For $v \in W_\ell(\omega) \setminus \{0\}$ the limit
			$	\lambda(\omega, v) = \lim_{n\to\infty}\frac{1}{n}\log\norm{v\rmatrix^{(n)}(\omega)}_1$
			exists and equals $\lambda_\ell$.
	\end{itemize}
	\label{thm:FLQmatrices}
\end{theorem}
\begin{lemma}
 Under the hypothesis of Theorem \ref{thm:FLQmatrices}, for all $\omega\in\tilde\Omega$ and all vectors $v>0$ one has $\lambda(\omega,v) = \lambda_1$. If, in addition, $\rmatrix$ is uniformly aperiodic, then for all $\omega\in\tilde\Omega$ and for all vectors $v\ge 0$ one has $\lambda(\omega,v) = \lambda_1$.
	\label{lem:positive_quadrant}
\end{lemma}
\begin{proof}
  Let $v_1\in\mathbb{R}^k$ satisfy $\lambda(\omega,v_1) = \lambda_1$. Since $|v_1|\rmatrix^{(n)}(\omega) \ge |v_1\rmatrix^{(n)}(\omega)|$, we must also have $\lambda(\omega, |v_1|) = \lambda_1$, where the absolute values and the inequalities are taken element-wise.
  Thus, the leading exponent $\lambda_1$ is achieved by a nonnegative vector, namely, $v'_1=|v_1|\ge 0$.


  Suppose first that in fact $v'_1>0$.
  For any $v>0$, there exist positive constants $c$ and $C$ such that $cv'_1 \le v \le Cv'_1$ and therefore for any $n\in\mathbb{N}$ we have $c\norm{v'_1\rmatrix^{(n)}(\omega)}_{1} \le \norm{v \rmatrix^{(n)}(\omega)}_{1} \le C\norm{v'_1\rmatrix^{(n)}(\omega)}_{1}$. We conclude that $\lambda(\omega,v) = \lambda(\omega, v'_1)$ for all positive $v$.
	
  Secondly, we consider the case where $v'_1$ is merely non-negative and non-zero.
    Since $\rmatrix$ is uniformly aperiodic, for every $\omega$ there exists an integer $N$ such that $\rmatrix^{(N)}(\omega)$ is positive and therefore $v'_1\rmatrix^{(N)}(\omega)$ is also positive. Using the argument above for positive vectors and the fact that $\lambda(\omega, v) = \lambda(\base^N\omega,v\rmatrix^{(N)}(\omega))$ we obtain $\lambda(\omega,v) = \lambda(\omega,v'_1)=\lambda_1$ for all $v\ge 0$.
\end{proof}

\begin{corollary}\label{cor:entropy}
	For all $\omega \in \tilde\Omega$ one has $h(\Sigma_\rmatrix(\omega)) = h(\Sigma_\rmatrix) = \lambda_1$.
\end{corollary}
\begin{proof}
  From Proposition \ref{prop:blocks}, clearly $\abs{\blocks_n(\omega)} = \norm{\mathbbm{1}\rmatrix^{(n-1)}(\omega)}_1$, thus $h(\Sigma_\rmatrix(\omega)) = \lambda(\omega,\mathbbm{1})$. By Lemma \ref{lem:positive_quadrant}, this equals $\lambda_1$ for all $\omega\in\tilde\Omega$.
\end{proof}

Our main result in this section is:
\begin{theorem}
  Let $\Sigma_\rmatrix$ be a uniformly aperiodic random shift of finite type with corresponding random adjacency matrix $\rmatrix:\Omega \to \mathcal{M}_{k\times k}(\{0,1\})$. Fix $\omega^*\in\tilde\Omega$. Let $v^{*}\in W_\ell(\omega^*)$ with $\ell>1$.
  Define the sequence of vectors $v : \{\base^n\omega^*\}_{n\in\mathbb{Z}^{+}} \to \mathbb{R}^{k}$ on the forward orbit of $\omega^{*}$ by
  \begin{displaymath}
    v(\base^n\omega^*) := \frac{v^*\rmatrix^{(n)}(\omega^{*})}{\norm{v^{*}\rmatrix^{(n)}(\omega^*)}_{1}}\in W_\ell(\base^n\omega^*)
  \end{displaymath}
  and a sequence of sub-alphabets $\mathcal{A}_{+}$ by $\mathcal{A}_{+}(\base^n\omega^*) := \{i\in\mathcal{A}\ :\ v_i(\base^n\omega^*) > 0 \}$. Suppose $\Sigma_{\submatrix}$ is a subshift of $\Sigma_{\rmatrix}$ such that on the orbit of $\omega^{*}$ the random matrix $\submatrix$ takes the following values:
	\begin{equation}\label{eq:B}
		\submatrix_{ij}(\base^n\omega^*) = \begin{cases}
		  \rmatrix_{ij}(\base^n\omega^*) & \textnormal{if } i\in \mathcal{A}_{+}(\base^n\omega^*)\textnormal{ and } j \in \mathcal{A}_{+}(\base^{n+1}\omega^*)\\
			0  & \textnormal{otherwise}. \end{cases}
	\end{equation}
	Then the topological entropy of $\Sigma_{\submatrix}(\omega^{*})$ is greater than or equal to $\lambda_{\ell}$, that is $h(\Sigma_\submatrix(\omega^*)) \ge \lambda_\ell$. If $\Sigma_{\submatrix'}$ is the complementary subshift to $\Sigma_{B}$ then $h(\Sigma_{\submatrix'}(\omega^*)) \ge \lambda_{\ell}$.
	\label{thm:rsmain}
\end{theorem}

\begin{remark}
  One may always find a random subshift $\Sigma_{\submatrix}$ as required in Theorem \ref{thm:rsmain}. For example, consider $\submatrix$ defined by
  \begin{displaymath}
    \submatrix_{ij}(\omega) =
    \begin{cases}
      \rmatrix_{ij}(\omega) & \textnormal{if } \omega\in\{\base^n\omega^*\}_{n\in\mathbb{Z}^{+}} \textnormal{ and } i\in\mathcal{A}_{+}(\omega),\ j\in\mathcal{A}_{+}(\base\omega)\\
      0 & \textnormal{otherwise}.
    \end{cases}
  \end{displaymath}
  It is then easy to see that the random matrix $\submatrix$ is measurable and that $\Sigma_{\submatrix}$ is a subshift of $\Sigma_{\rmatrix}$.
\end{remark}

\begin{proof}[Proof of Theorem \ref{thm:rsmain}]
  Firstly, we will show by induction that for all $n\ge 1$ and all $i\in\mathcal{A}_{+}(\base^n\omega^{*})$
\begin{equation}
	(v(\omega^*)\rmatrix^{(n)}(\omega^*))_i \le (v(\omega^*)\submatrix^{(n)}(\omega^*))_i.
	\label{eq:induction1}
\end{equation}
Let $v = v^{+} + v^{-}$ denote the decomposition of the vector $v$ into nonnegative and nonpositive parts. Then we have $(v(\omega^*)\rmatrix(\omega^*))_i \le (v(\omega^*)^+\rmatrix(\omega^*))_i = (v(\omega^*)\submatrix(\omega^*))_i$ so \eqref{eq:induction1} holds for $n=1$. Assuming that \eqref{eq:induction1} is true for some $n\ge 1$, we proceed with the inductive step
\begin{eqnarray}
  \nonumber (v(\omega^*)\rmatrix^{(n+1)}(\omega^*))_i&=&  \sum_{j} (v(\omega^*)M^{(n)}(\omega^*))_jM_{ji}(\base^n\omega^*)\\
  \nonumber &\le&  \sum_{j\in\mathcal{A}_{+}(\base^n\omega^*) } (v(\omega^*)M^{(n)}(\omega^*))_jM_{ji}(\base^n\omega^*)\\
  \nonumber &\le&  \sum_{j\in\mathcal{A}_{+}(\base^n\omega^*) } (v(\omega^*)Q^{(n)}(\omega^*))_jM_{ji}(\base^n\omega^*)\quad\mbox{by hypothesis}\\
  \nonumber &=&  \sum_{j\in\mathcal{A}_{+}(\base^n\omega^*) } (v(\omega^*)Q^{(n)}(\omega^*))_jQ_{ji}(\base^n\omega^*)\quad\mbox{as $j\in \mathcal{A}_{+}(\base^n\omega^*), i\in \mathcal{A}_{+}(\base^{n+1}\omega^*)$}\\
  \nonumber &=&(v(\omega^*)Q^{(n+1)}(\omega^*))_i
  %
%
\end{eqnarray}
Thus \eqref{eq:induction1} holds for all $n\ge 1$ and all $i\in\mathcal{A}(\base^n\omega^*)$. Since  both sides of \eqref{eq:induction1} are positive, we have
\begin{displaymath}
	\norm{(v(\omega^*)\rmatrix^{(n)}(\omega^*))^+}_1 \le \norm{v(\omega^*)\submatrix^{(n)}(\omega^*)}_1.
\end{displaymath}
Thus,
\begin{align*}
	\lim_{n\to\infty}\frac{1}{n}\log \norm{v(\omega^*)\rmatrix^{(n)}(\omega^*))^+}_1 &\le \lim_{n\to\infty}\frac{1}{n}\log\norm{v(\omega^*)\submatrix^{(n)}(\omega^*)}_1\\
	&\le \lim \frac{1}{n}\log \norm{\mathbbm{1}\submatrix^{(n)}(\omega^*)}_1	= h(\Sigma_\submatrix(\omega^*)).
\end{align*}
Next we will show that $\lim_n (1/n) \log\norm{v(\omega^*)\rmatrix^{(n)}(\omega^*)}_{1}\le\lim_n (1/n) \log\norm{(v(\omega^*)\rmatrix^{(n)}(\omega^*))^+}_1$ to finally obtain that $\lambda_\ell\le h(\Sigma_\submatrix(\omega^*))$.
We need to have some control over the relative size of the positive and negative parts of $v$ along the orbit of $\omega^*$. To continue the proof of Theorem \ref{thm:rsmain} we first state and prove the following claim.

{\noindent \it Claim: Let $\omega\in\{\base^n\omega^*\}_{n\in\mathbb{Z}^{+}}$ and let $N=N(\omega)$ be smallest integer such that $\rmatrix^{(N)}(\omega) > 0$. Then}
	\begin{equation}
	  \frac{1}{k^N}\le \frac{\norm{v(\omega)^+}_{1}}{\norm{v(\omega)^-}_{1}} \le k^N.
		\label{eq:vector_bound}
	\end{equation}
	\label{lem:bound}

\begin{proof}[Proof of claim:]
  As $\rmatrix$ is a $0-1$ matrix, then $\max_{i,j} \rmatrix_{ij}^{(N)}(\omega) \le k^N$. From the definition of $N$ we also have $\min_{i,j}\rmatrix_{ij}^{(N)}(\omega) \ge 1$. The proof of \eqref{eq:vector_bound} is by contradiction. Suppose that $\norm{v(\omega)^+}_{1} > k^N\norm{v(\omega)^-}_{1}$. Then for every $i \in \mathcal{A}$ we have
	\begin{align*}
	  v_{i}(\base^N\omega)\norm{v(\omega)\rmatrix^{(N)}(\omega)}_{1} &= (v(\omega)\rmatrix^{(N)}(\omega))_i\\
	  &= (v(\omega)^+\rmatrix^{(N)}(\omega) + v(\omega)^-\rmatrix^{(N)}(\omega))_i\\
		&= \sum_j (v(\omega)^+)_j\rmatrix^{(N)}_{ji}(\omega) + \sum_j (v(\omega)^-)_j\rmatrix^{(N)}_{ji}(\omega)\\
		& \ge \norm{v(\omega)^+}_{1} - k^N\norm{v(\omega)^-}_{1}\\
		& > k^N\norm{v(\omega)^-}_{1} - k^N\norm{v(\omega)^-}_{1} = 0.
	\end{align*}
        Therefore $v(\base^{N}\omega) \in W_{\ell}(\base^N\omega)$ is a positive vector, but this is a contradiction because the Lyapunov exponent of any positive vector equals $\lambda_1 \neq \lambda_\ell$.
	The inequality $1/k^N \le \norm{v(\omega)^+}_{1}/\norm{v(\omega)^-}_{1}$ is proven similarly.
\end{proof}

We continue the proof of the theorem as follows
\begin{eqnarray*}
	\lambda_\ell &=& \lim_{n\to\infty} \frac{1}{n} \log\norm{(v(\omega^*)\rmatrix^{(n)}(\omega^*))}_1\\
	&=&\lim_{n\to\infty} \frac{1}{n} \log\left(\norm{(v(\omega^*)\rmatrix^{(n)}(\omega^*))^+}_1+\norm{(v(\omega^*)\rmatrix^{(n)}(\omega^*))^-}_1\right)\\
	&\le& \lim_{n\to\infty}\frac{1}{n} \log((1+k^N)\norm{(v(\omega^*)\rmatrix^{(n)}(\omega^*))^+}_1)\\
	&=& \lim_{n\to\infty} \frac{1}{n} \log\norm{(v(\omega^*)\rmatrix^{(n)}(\omega^*))^+}_1.
\end{eqnarray*}
\end{proof}
Theorem \ref{thm:rsmain} may be used to decompose a metastable random shift space into two complementary random subshifts, each with large topological entropy along the orbit of a chosen $\omega^{*}\in\Omega$. One chooses $v^*= v(\omega^{*})\in W_{2}(\omega^*)$ corresponding to the second largest Lyapunov exponent $\lambda_2$ and then partitions according to the positive and negative parts of the push-forwards of $v$ by the random matrix $\rmatrix$. We illustrate this with the following example.

\begin{example}
  Let $\Omega = \{0,1\}^{\mathbb{Z}}$ and let $\base:\Omega\circlearrowleft$ be the full two sided shift on two symbols. Consider the random matrix $\rmatrix : \Omega \to \mathcal{M}_{4\times 4}$ given by $\rmatrix(\omega) = \rmatrix_{\omega_{0}}$ where
  \begin{displaymath}
    \rmatrix_{0} = \left[
\begin{array}{rrrr}
0  &  1  &  0  &  0\\
     1  &  1  &  0  &  1\\
     0  &  0  &  1  &  1\\
     1  &  0  &  1  &  0
\end{array}
\right] \quad \text{and }
\rmatrix_{1} = \left[
\begin{array}{rrrr}
1  &  1  &  0  &  0\\
     1  &  1  &  1  &  0\\
     0  &  1  &  1  &  1\\
     0  &  0  &  1  &  1
\end{array}
\right].
  \end{displaymath}
  We consider a generic point $\omega^*\in\tilde\Omega$ where $\omega^*_{i}$ is the $(20+i)^{th}$ digit of the fractional part of the binary expansion of $\pi$ for $i>-20$ and ($\omega^*_{i} = 0$ for $i\le-20$). The first few elements of $\omega^{*}$, with the zeroth element underlined, are given below:

  \begin{displaymath}
     \omega^{*} = (\dots,1,0,0,0,1,0,0,0,0,1,\underline{0},1,1,0,1,0,0,0,1,1,0,\dots).
  \end{displaymath}

  Using the algorithm on p745 \cite{FLQ09} (with $M = 2N = 40$) we approximate the largest Lyapunov exponent $\lambda_1(\omega^*)\approx\log2.20$, which equals the orbit topological entropy of the random shift, that is $h(\Sigma_{\rmatrix}(\omega^*)) \approx \log 2.20$. The second Lyapunov exponent along the forward orbit of $\omega^{*}$ is $\lambda_2(\omega^{*}) \approx \log 1.21$. Thus, by Theorem \ref{thm:rsmain} we can decompose the shift $\Sigma_\rmatrix$ into two complementary subshifts, each with topological entropy along the orbit of $\omega^{*}$ larger than $\log 1.21$.  Moreover, the decomposition is given by the Oseledets subspaces for $\lambda_2(\omega^{*})$. These Oseledets subspaces are spans of the vectors $\{v(\omega^{*}), v(\base\omega^{*}), v(\base^2\omega^{*}),\dots \}$, whose graphs are shown in Figure \ref{fig:Subspaces}.
  The sub-alphabets $\mathcal{A}_{+}$ and $\mathcal{A}_{-}$ have the following values on the first four points in the forward orbit of $\omega^{*}$:
  \begin{displaymath}
    \begin{array}{llcll}
      \mathcal{A}_{+}(\omega^{*}) &= \{1,2\}, & \quad & \mathcal{A}_{-}(\omega^{*}) &=\{3,4\},\\
      \mathcal{A}_{+}(\base\omega^{*}) &= \{2,4\}, & \quad & \mathcal{A}_{-}(\base\omega^{*}) &=\{1,3\},\\
      \mathcal{A}_{+}(\base^2\omega^{*}) &= \{1,3\}, & \quad & \mathcal{A}_{-}(\base^2\omega^{*}) &=\{2,4\},\\
      \mathcal{A}_{+}(\base^3\omega^{*}) &= \{1,2\}, & \quad & \mathcal{A}_{-}(\base^3\omega^{*}) &=\{3,4\}.
    \end{array}
  \end{displaymath}

\begin{figure}[htb]
    \psfrag{vomega}{\small$v(\omega^{*})$}
    \psfrag{vomega1}{\small$v(\base\omega^{*})$}
    \psfrag{vomega2}{\small$v(\base^2\omega^{*})$}
    \psfrag{vomega3}{\small$v(\base^3\omega^{*})$}
    \begin{center}
      \includegraphics[width=13cm]{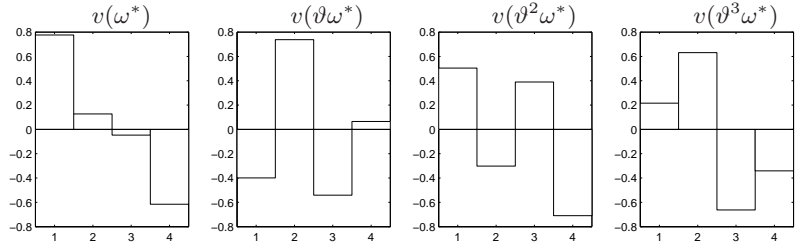}
    \end{center}
\caption{Vectors spanning Oseledets subspaces corresponding to the second Lyapunov exponent.}
    \label{fig:Subspaces}
\end{figure}

We construct the matrix $\submatrix$ on the orbit of $\omega^{*}$ of the random subshift according to \eqref{eq:B} in Theorem \ref{thm:rsmain}.
  The first three values of $\submatrix$  along the orbit are given below:

   {\small \begin{displaymath}
   \submatrix(\omega^*) = \left[ \begin{array}{rrrr}
     0  &  1  &  0  &  0\\
     0  &  1  &  0  &  1\\
     0  &  0  &  0  &  0\\
     0  &  0  &  0  &  0
    \end{array} \right], \ \
    \submatrix(\base\omega^*) = \left[ \begin{array}{rrrr}
     0  &  0  &  0  &  0\\
     1  &  0  &  1  &  0\\
     0  &  0  &  0  &  0\\
     0  &  0  &  1  &  0
    \end{array} \right],\ \
    \submatrix(\base^2\omega^*) = \left[ \begin{array}{rrrr}
     1  &  1  &  0  &  0\\
     0  &  0  &  0  &  0\\
     0  &  1  &  0  &  0\\
     0  &  0  &  0  &  0
    \end{array} \right].
  \end{displaymath}}

  Similarly the adjacency matrices of the complementary subshift $\Sigma_{\submatrix'}$ begin:
  \small{\begin{displaymath}
    \submatrix'(\omega^*) = \left[ \begin{array}{rrrr}
     0  &  0  &  0  &  0\\
     0  &  0  &  0  &  0\\
     0  &  0  &  1  &  0\\
     1  &  0  &  1  &  0
    \end{array} \right], \ \
    \submatrix'(\base\omega^*) = \left[ \begin{array}{rrrr}
     0  &  1  &  0  &  0\\
     0  &  0  &  0  &  0\\
     0  &  1  &  0  &  1\\
     0  &  0  &  0  &  0
    \end{array} \right],\ \
    \submatrix'(\base^2\omega^*) = \left[ \begin{array}{rrrr}
     0  &  0  &  0  &  0\\
     0  &  0  &  1  &  0\\
     0  &  0  &  0  &  0\\
     0  &  0  &  1  &  1
    \end{array} \right].
  \end{displaymath}}

The graphs of $\Sigma_\submatrix$ and $\Sigma_{\submatrix'}$ for the first four elements of the forward orbit of $\omega^*$ are shown in Figure \ref{fig:subgraph}.

\begin{figure}[hbt]
    \begin{center}
   \subfigure[Graph of $\Sigma_{\submatrix}$.]{
      \includegraphics[width=5.5cm]{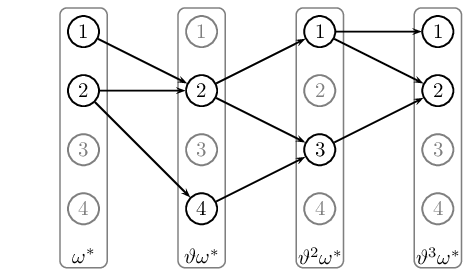}
      \label{fig:B}}\hspace{1cm}
      \subfigure[Graph of $\Sigma_{\submatrix'}$.]{
      \includegraphics[width=5.5cm]{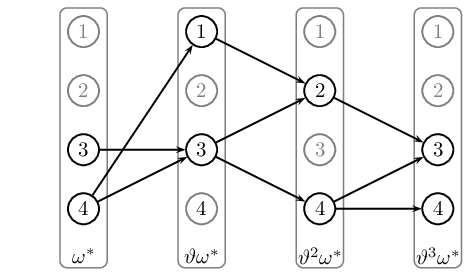}
      \label{fig:Bdash}}
    \end{center}
    \caption{Graphs of $\Sigma_{\submatrix}$ and $\Sigma_{\submatrix'}$ for first four transitions along the orbit of $\omega^{*}$. The grayed-out nodes in each belong to the corresponding complementary subshift.}
    \label{fig:subgraph}
  \end{figure}
 Using the algorithm on p745 in \cite{FLQ09} (with $M=20$ and $N=0$) we estimate the largest Lyapunov exponents, and hence the topological entropies, of these two subsifts on the forward orbit of $\omega^{*}$ to be $h(\Sigma_{\submatrix}(\omega^*))\approx \log 1.62$ and $h(\Sigma_{\submatrix'}(\omega^*))\approx \log1.58$. Both are larger than $\lambda_2\approx\log1.21$, as predicted by Theorem \ref{thm:rsmain}.
\end{example}

\appendix
\section{Grassmannians}
Let $(Y, \norm{\cdot}_{Y})$ be a Banach space. A subspace $E$ of $Y$ is said to be \emph{closed complemented} if it is closed and there exists a closed subspace $F$ of $Y$ such that $E\cap F = \{0\}$ and $E + F = Y$, where `$+$' denotes the direct sum; that is, any non-zero element of $Y$ can be uniquely written as $e+f$ with $e\in E$ and $f\in F$. A natural linear map on $Y$ is the \emph{projection} onto $F$ along $E$, defined by $\proj_{F\parallel E}(e+f) = f$.

The \emph{Grassmannian} $\mathcal{G}(Y)$ of the space $Y$ is the set of all closed complemented subspaces of $Y$. For any $E_0\in\mathcal{G}(Y)$ there exists at least one $F_0\in\mathcal{G}(Y)$ such that $E_0 \oplus F_0 = Y$, where `$\oplus$' now denotes topological direct sum. Every such $F_0$ defines a \emph{neighbourhood} $U_{F_0}(E_0)$ of $E_0$ by $U_{F_0}(E_0) := \{E\in\mathcal{G}(Y) \ : \ E\oplus F_0 = Y \}$. Furthermore, on every such neighbourhood we can define a $(E_0, F_0)$-\emph{local norm} by
\begin{displaymath}
	\norm{E}_{(E_0,F_0)} := \norm{\proj_{F_0\parallel E}|_{E_0}}_{op}.
\end{displaymath}
This induces a topological structure of a Banach manifold on $\mathcal{G}(Y)$. In particular, given a suitable topology on $\Omega$, the continuity of maps $\Omega\to\mathcal{G}(Y)$ is well-defined. In a similar fashion, by taking the corresponding Borel $\sigma$-algebra $\Borel(\mathcal{G}(Y))$ and $\Borel(\Omega)$ (or $\F$), we may also define measurability of such maps.

Recall that a function $f : Y\to \R$ is said to be \emph{upper semi-continuous} at $x_0\in Y$ if for every $\epsilon > 0$ there is an open neighbourhood $U_{x_0}$ such that $f(x) \le f(x_{0}) + \epsilon$ for all $x\in U_{x_0}$. The following lemma is used in the proof of Theorem \ref{thm:equiv_norms}.
\begin{lemma}
  Let $d$ be a fixed integer. Let $(Y,\norm{\cdot}_{Y})$ be a Banach space and $\norm{\cdot}_{*}$ a second norm on $Y$ such that $\norm{\cdot}_{*} \le \norm{\cdot}_{Y}$. The function $\psi: \mathcal{G}_d(Y) \to \mathbb{R}$ defined as $\psi(E) = \sup_{\xi\in E} \norm{\xi}_{Y}/\norm{\xi}_{*}$ is upper semi-continuous and therefore measurable.
	\label{lem:usc}
\end{lemma}
\begin{proof}
	Each $E\in\mathcal{G}_{d}(Y)$ is finite-dimensional. Since all norms on finite dimensional spaces are equivalent, the function $\psi$ is well defined and $1\le \psi(E) < \infty$ for all $E\in\mathcal{G}_d(Y)$. For any $E_0 \in \mathcal{G}_{d}(Y)$, let $F_{0} \in \mathcal{G}^{d}(Y)$ be such that $E_{0} \oplus F_{0} = Y$. For any $\epsilon \in (0, \psi(E_0)^{-1})$ let $N_{\epsilon} \subset U_{F_0}(E_0)$ be a neighbourhood of $E_{0}$ such that for all $E\in N_{\epsilon}$,
	\begin{displaymath}
	\norm{E}_{(E_0,F_0)} = \norm{\proj_{F_0\parallel E}|_{E_0}}_{op} < \epsilon.
	\end{displaymath}
        Take any $E\in N_\epsilon$. For any $x\in E$ write $x = y - z$ where $y\in E_0$ and $z\in F_0$. Then $z = \proj_{F_0\parallel E}(y)$ and $\norm{z}_{*}/\norm{y}_{Y} \le \norm{z}_{Y}/\norm{y}_{Y} < \epsilon$. Now
	\begin{displaymath}
          \frac{\norm{x}_{Y}}{\norm{x}_{*}} = \frac{\norm{y + x - y}_{Y}}{\norm{y+x-y}_{*}}
          \le \frac{\norm{y}_{Y} + \norm{z}_{Y}}{\norm{y}_{*} - \norm{z}_{*}}
          < \frac{\norm{y}_{Y} + \epsilon\norm{y}_{Y}}{\norm{y}_{*} - \epsilon\norm{y}_{Y}}
		\le \frac{1+ \epsilon}{\psi(E_0)^{-1} - \epsilon}.
	\end{displaymath}
	The right hand side converges to $\psi(E_{0})$ as $\epsilon\to 0$. As $E_0$ and $\epsilon$ are arbitrary, this establishes upper semi-continuity of $\psi$ on all of $\mathcal{G}_{d}(Y)$.
\end{proof}

\newcommand{\etalchar}[1]{$^{#1}$}

\end{document}